\documentclass{amsart}

\usepackage{amssymb,amsmath,esint}
\usepackage[all]{xy}
\usepackage{graphicx}
\usepackage{enumerate}

\usepackage{colortbl}
\usepackage{cite}
\usepackage{xcolor}

\usepackage{amsrefs}

\usepackage{verbatim}

\usepackage{geometry}
\usepackage{comment}
\usepackage{graphicx,caption,subcaption}
\usepackage{tikz}
\usepackage{multirow}

\usepackage[title,titletoc]{appendix}

\colorlet{linkequation}{blue}
\usepackage[colorlinks,plainpages=true,pdfpagelabels,hypertexnames=true,colorlinks=true,pdfstartview=FitV,linkcolor=blue,citecolor=red,urlcolor=black]{hyperref}
\usepackage{hyperref}

\PassOptionsToPackage{unicode}{hyperref}
\PassOptionsToPackage{naturalnames}{hyperref}
\definecolor{dgreen}{rgb}{0,0.5,0}
\definecolor{violet}{rgb}{0.5,0,0.5}
\definecolor{dred}{rgb}{0.7,0,0}
\definecolor{ddred}{rgb}{0.5,0,0}
\definecolor{dblue}{rgb}{0,0,0.5}
\definecolor{ddblue}{rgb}{0,0,0.3}
\definecolor{llgray}{rgb}{0.9,0.9,0.9}
\definecolor{lgray}{rgb}{0.7,0.7,0.7}
\definecolor{dgray}{rgb}{0.2,0.2,0.2}

\newtheorem{defn}{Definition}[section]
\newtheorem{definition}[defn]{Definition}
\newtheorem{lemma}[defn]{Lemma}

\newtheorem{theorem}[defn]{Theorem}

\newtheorem{remark}[defn]{Remark}
\newtheorem{example}[defn]{Example}
\numberwithin{equation}{section}

\newcommand{\bq}{\begin{equation}}
\newcommand{\eq}{\end{equation}}

\newcommand{\R}{{ \mathbb{R}  }}

\newcommand{\abs}[1]{\left| #1 \right|}

\AtBeginDocument{%
   \def\MR#1{}
}
\nocite{*}

\begin{document}

\title[Existence of PME with measure and drifts]{Existence of weak solutions for nonlinear drift-diffusion equations with measure data}
%Existence of weak solutions for nonlinear diffusion equations with drift for measure data
%
%Existence of weak solutions for nonlinear diffusion equations with measure data and divergence type of drift terms

\author[S. Hwang]{Sukjung Hwang}
\address{S. Hwang: Department of Mathematics Education, Chungbuk National University, Cheongju 28644, Republic of Korea}
\email{sukjungh@cbnu.ac.kr}

\author[K. Kang]{Kyungkeun Kang}
\address{K. Kang: Department of Mathematics, Yonsei University,  Seoul 03722, Republic of Korea}
\email{kkang@yonsei.ac.kr}

\author[H.K. Kim]{Hwa Kil Kim}
\address{H.K. Kim: Department of Mathematics Education, Hannam University, Daejeon 34430, Republic of Korea}
\email{hwakil@hnu.kr}

\author[J.-T. Park]{Jung-Tae Park}
\address{J.-T. Park: School of Liberal Arts, Korea University of Technology and Education, Cheonan 31253, Republic of Korea; and Korea Institute for Advanced Study, Seoul 02455, Republic of Korea}
\email{jungtae.park@koreatech.ac.kr}

\thanks{
S. Hwang's work is partially supported by RS-2025-24523482. K. Kang is supported by NRF grant nos. RS-2024-00336346 and RS-2024-00406821. H. Kim's work is partially supported by NRF-2021R1F1A1048231. J.-T. Park is supported by Education and Research promotion program of KOREATECH in 2024 and RS-2025-25433258.}

\date{May 28, 2026}
%\date{June 20, 2025}

%%%%%%%%%%%%%%%%%%%
\makeatletter
\@namedef{subjclassname@2020}{%
  \textup{2020} Mathematics Subject Classification}
\makeatother
%{2020 AMS Subject Classification}\, :\,  35A01, 35K55, 35Q84, 92B05

%{Keywords}\,:\, Porous medium equation, weak solution, Wasserstein space, a bounded domain
%\Addresses

%%%%%%%%%%%%%%%%%%%%%%%%%%%%%%%%%%%%%%%%%%%%%%%%%%%%%%%%%

\begin{abstract}
We consider nonlinear drift-diffusion equations (both porous medium equations and fast diffusion equations) with measure data. We establish the existence of nonnegative weak solutions satisfying gradient estimates, provided that the drift term belongs to a sub-scaling class relevant to the $L^1$ space. When the drift is divergence-free, this requirement can be relaxed: the drift may belong to a class that is supercritical with respect to $L^1$-scaling class, and the admissible range of the diffusion exponent $m$ is enlarged as well. By handling both the measure data and the drift, we obtain a new type of energy estimate. We also discuss sharpness by constructing counterexamples showing that the general-drift range cannot be improved under the corresponding integrability scale without the divergence-free cancellation. As an application, we construct weak solutions for a specific type of nonlinear diffusion equation with measure data coupled to the incompressible Navier-Stokes equations.
\end{abstract}
\maketitle

%%%%%%%%%%%%%%%%%%%%%%%%%%%%%%%%%%%%%%%%%%%%%%%%%%%%%%%%%
%{
%  \hypersetup{linkcolor=black}
%  \tableofcontents
%}

%\hypersetup{linkcolor=black}
%\tableofcontents

%%%%%%%%%%%%%%%%%%%%%%%%%%%%%%%%%%%%%%%%%%%%%%%%%%%%%%%%%%%%

\section{Introduction}
In this paper, we study the following form of nonlinear diffusion equations, which include a drift term in divergence form and measure data:
\begin{equation}\label{PME}
	u_t - \Delta u^m + \nabla \cdot \left( u \, V \right)= \mu, \ \text{ for } \  m>0 \ \text{ on } \ \Omega_{T},
\end{equation}
where \(\mu\) is a nonnegative finite Radon measure in $\Omega_T$, and \(\Omega_T := \Omega \times (0,T)\) for a bounded domain \(\Omega \subset \mathbb{R}^d\), \(d \geq 2\). We impose zero boundary conditions on both the initial and the lateral boundaries. We assume that $\mu$ is defined in $\mathbb{R}^{d+1}$ by considering the zero extension to $\mathbb{R}^{d+1}$.
The nonlinear diffusion equation is called the porous medium equation when $m>1$, heat equation when $m=1$, and fast diffusion equation when $0<m<1$. These equations are important to describe various natural and physical phenomena, see \cite{DK07, DGV12, Vaz06, Vaz07} and references therein.

Our main objective is to find sufficient conditions on the drift vector field
\(V\) which guarantee the existence of nonnegative weak solutions to
\eqref{PME}. The main feature of the problem is that it combines two
different sources of difficulty: the low regularity of the measure-valued
forcing term and the presence of the drift term. Each of these difficulties
has been studied separately, but their simultaneous presence changes the
compactness mechanism in an essential way.

We first recall the drift-diffusion problem without measure data. If
\(\mu=0\), then \eqref{PME} reduces to
\begin{equation}\label{PME_drift}
	u_t-\Delta u^m+\nabla\cdot(uV)=0.
\end{equation}
This type of equation appears in drift-diffusion models, including
Keller--Segel-type systems
\cite{BedHe2017,Freitag,KK-SIMA,Win15,CCY2019}. Critical and subcritical
conditions on \(V\) also play an important role in fluid dynamics and related
problems; see \cite{SVZ13,SSSZ,Z11}. Various existence and regularity results
for weak solutions have been obtained under suitable assumptions on \(V\);
see, for instance,
\cites{KZ18,KZ2021,HZ21,HKK01,HKK02,HKK_FDE}.

When a measure-valued forcing term is present, however, the approaches used
for \eqref{PME_drift} are no longer directly applicable. Standard energy
estimates cannot be used in the same form, and one has to work with
truncated energy estimates adapted to finite Radon measures. As a
consequence, the admissible classes of drift fields become more restrictive
than in the case without measure data; see Remarks~\ref{R:T:PME} and
\ref{R:T:divfree}.

On the other hand, when the drift is absent, there is a well-developed
existence theory for nonlinear diffusion equations with measure-valued
right-hand sides. Lukkari first proved existence results for the model
equation
\(
u_t-\Delta u^m=\mu
\)
in \cite{Luk10} for the porous medium case \(m>1\), and in \cite{Luk12} for
the fast diffusion case
\(
\left(1-\frac{2}{d}\right)_+<m<1.
\)
We also refer to \cite{KLLP19}. Later, B\"ogelein, Duzaar, and Gianazza
extended these results to more general nonlinear diffusion problems; see
\cite{BDG13} for the degenerate case \(m>1\), and \cite{BDG15} for the
singular case
\(
\left(1-\frac{2}{d}\right)_+<m<1.
\)
Further boundedness and continuity results via parabolic Riesz potentials
can be found in \cite{BDG13,BDG14,LS13,Stu15} for \(m>1\), and in
\cite{BDG15,BDG16} for
\(
\left(1-\frac{2}{d}\right)_+<m<1.
\)

The lower bound
\(
m>\left(1-\frac{2}{d}\right)_+
\)
is not merely technical in the measure data setting. Even in the drift-free case \(V=0\), Barenblatt solutions show that this threshold is intrinsic to nonlinear diffusion equations with finite-measure right-hand sides. More precisely, the Barenblatt solution associated with
\(
\partial_t u-\Delta u^m=\delta_{(0,0)},
\)
where \(\delta_{(0,0)}\) denotes the Dirac mass at the space-time origin, is available as a solution precisely in the range corresponding
to
\(
m>\left(1-\frac{2}{d}\right)_+;
\)
see \cite{BDG13,BDG15}. Thus the lower bound appearing in the measure data theory is already optimal in the absence of drift.

The equation studied in the present paper lies at the interface of the two theories above. It contains both a divergence-form drift term and a measure-valued forcing term. Therefore, it is not a straightforward extension of either the drift-diffusion theory without measures or the measure data theory without drift. In the general drift case, the local truncation estimates contain additional terms involving the product \(uV\). These terms cannot be controlled by the measure-data compactness method alone, and the compactness methods of \cite{BDG13,BDG15} cannot be applied directly.

For this reason, in the general drift case we use a direct compactness argument based on estimates for \(u\). This argument naturally leads to the range
\[
1-\frac{1}{d}<m\le2.
\]
The lower bound on \(m\) provides enough Sobolev compactness to apply the Aubin--Lions lemma. The upper restriction \(m\le2\) is connected with the same direct compactness mechanism: when one tries to recover compactness of \(u\) from estimates on \(\nabla u^{\frac{m}{2}}\), the factor
\(
u^{1-\frac{m}{2}}
\)
appears, and this factor becomes singular near the zero set when \(m>2\). Under suitable assumptions on \(V\), this gives the existence result stated in Theorem~\ref{T:PME}.

We also show that the lower threshold
\(
m>1-\frac{1}{d}
\)
is sharp for general, non-divergence-free drifts. In
Section~\ref{S:sharpness-divfree}, we construct, in dimension \(d=2\), a
non-divergence-free drift field satisfying the same isotropic integrability
scale as in the divergence-free theory, together with a nonnegative finite
Radon measure on the right-hand side, for which the gradient estimate fails
when
\(
\left(1-\frac{2}{d}\right)_+<m\le 1-\frac{1}{d}.
\)
Thus, below the general-drift threshold, the estimates obtained in the
divergence-free case cannot be expected for arbitrary drift fields under the
same integrability assumptions on \(V\).

The same section also records a further obstruction in the highly degenerate range \(m>2\). It shows that without the divergence-free structure, a general drift may cancel the diffusion term and destroy the gradient estimate for \(u^{\frac{m}{2}}\), even when the drift belongs to an admissible isotropic integrability scale. This illustrates why the upper restriction \(m\le2\) appears naturally in the direct compactness argument for general drifts, and why the divergence-free cancellation is essential in the range \(m>2\) as well.

The situation improves substantially when \(V\) is divergence-free. In this case, the drift term has a cancellation structure in the local truncation estimates. More precisely, the troublesome drift contribution can be reduced to a lower-order cutoff term, which can be controlled locally. This allows us to adapt the compactness methods of \cite{BDG13,BDG15} to the present equation with divergence-free drift. Consequently, the assumptions on \(V\) are relaxed, and the lower bound on \(m\) is improved to
\(
m>\left(1-\frac{2}{d}\right)_+;
\)
see Theorem~\ref{T:divfree}. This distinction between the general drift case and the divergence-free case is one of the main contributions of the paper.

In this way, the two lower bounds have different meanings. The lower bound
\(
m>\left(1-\frac{2}{d}\right)_+
\)
comes from the measure data problem itself, already in the drift-free case. On the other hand, the stricter condition
\(
m>1-\frac{1}{d}
\)
is sharp for general drift fields. The divergence-free condition is precisely the structural assumption that restores the optimal measure-data range in the presence of a drift term. Moreover, in the highly
degenerate range \(m>2\), the divergence-free cancellation also prevents the
drift from destroying the gradient estimate by canceling the diffusion.

As an application of our existence theory, we study a simplified Keller--Segel--fluid type system with nonlinear diffusion and measure data. In this application, the velocity field is divergence-free, and our divergence-free existence theorem applies naturally. We construct weak solutions satisfying a global energy inequality; see Section~\ref{S:Application}.

We now describe the organization of the paper. In Section~\ref{S: Main}, we state the main existence theorems, classified according to the nonlinear diffusion regime and the assumptions on the drift field. In Section~\ref{S:Preliminaries}, we recall preliminary results used throughout the paper. A priori estimates are derived in Section~\ref{S:apriori-est}. The existence of weak solutions is proved in Section~\ref{S:Exist}. In Section~\ref{S:sharpness-divfree}, we show the sharpness of the general drift threshold by constructing counterexamples in dimension \(d=2\). Finally, in Section~\ref{S:Application}, we study the application to a nonlinear diffusion equation with measure data coupled to a fluid equation.

%-------------------------------------
\subsection{Main results}\label{S: Main}

Before stating our main existence results, we introduce the notion of weak solutions to \eqref{PME} and the classes of vector fields $V$ used in the construction. We then classify the existence results according to the diffusion regime and the assumptions on the drift.

\begin{definition}\label{D:WS}
Let $m>0$, $d\geq 2$, and $V:\Omega_{T} \to \mathbb{R}^d$ be a measurable vector field. We say that $u$ is a nonnegative weak solution of \eqref{PME} in $\Omega_{T}$ with zero initial data and zero lateral boundary data if the following hold:
\begin{enumerate}[(i)]
  \item It holds that
	\[
	u,\ \nabla u^m,\ uV \in L_{x,t}^{1}(\Omega_T).
	\]
  \item For every $\varphi \in C^{\infty} \left(\overline{\Omega} \times [0, T]\right)$ which vanishes on $\partial \Omega \times (0,T)$ and $\varphi(\cdot,T)=0$, we have
	\[
	\int_{\Omega_{T}} \left\{ - u \varphi_t + \nabla u^m \cdot\nabla \varphi - u \, V\cdot \nabla \varphi \right\} \,dxdt = \int_{\Omega_{T}} \varphi \,d\mu.
\]
\end{enumerate}
\end{definition}

We now introduce the suitable function classes for \(V\) that will be employed in the construction of weak solutions. In the divergence-free case, we adopt less restrictive classes.
\begin{definition}\label{Def:V}
For $m>0$ and $d\ge 2$, we define the following spaces, for $q_1, \, q_2 \in [1, \infty]$:
\begin{enumerate}[(i)]
  \item Let $m > 1-\frac{1}{d}$. The scaling invariant class is defined as
	\begin{equation}\label{S}
		\mathcal{S}_{m}^{(q_1, q_2)} := \left\{ V: \ \|V\|_{L_{x,t}^{q_1, q_2}(\Omega_T)} < \infty, \text{ where } \frac{d}{q_1} + \frac{2+d(m-1)}{q_2} = 1+ d(m-1)\right\},
	\end{equation}
	and the scaling subclass is defined as
	\begin{equation}\label{sub_S}
		\mathfrak{S}_{m}^{(q_1, q_2)} := \left\{ V: \ \|V\|_{L_{x,t}^{q_1, q_2}(\Omega_T)} < \infty, \text{ where } \frac{d}{q_1} + \frac{2+d(m-1)}{q_2} < 1+ d(m-1)\right\}.
	\end{equation}
	 Moreover, let us denote $\|V\|_{\mathcal{S}_{m}^{(q_1,q_2)}}$ and $\|V\|_{\mathfrak{S}_{m}^{(q_1, q_2)}}$ as the norms corresponding to each space.
  \item Let $m > (1-\frac{2}{d})_{+}$. The $\sigma$-class is defined as
	\begin{equation}\label{S_divfree}
		\mathcal{S}_{m,\sigma}^{(q_1, q_2)} := \left\{ V: \ \nabla\cdot V = 0 \ \text{ and } \ \|V\|_{L_{x,t}^{q_1, q_2}(\Omega_T)} < \infty, \text{ where } \frac{d}{q_1} + \frac{2+d(m-1)}{q_2} = 2+ d(m-1)\right\},
	\end{equation}
	and the $\sigma$-subclass is defined as
	\begin{equation}\label{sub_S_divfree}
		\mathfrak{S}_{m,\sigma}^{(q_1, q_2)} := \left\{ V: \ \nabla\cdot V = 0 \ \text{ and } \ \|V\|_{L_{x,t}^{q_1, q_2}(\Omega_T)} < \infty, \text{ where } \frac{d}{q_1} + \frac{2+d(m-1)}{q_2} < 2+ d(m-1)\right\}.
	\end{equation}
	Moreover, let us denote $\|V\|_{\mathcal{S}_{m, \sigma}^{(q_1, q_2)}}$ and $\|V\|_{\mathfrak{S}_{m,\sigma}^{(q_1, q_2)}}$ as the norms corresponding to each space.
\end{enumerate}
\end{definition}

Let us remark on the classes of \(V\) introduced in Definition~\ref{Def:V}:
\begin{remark}
\begin{enumerate}[(i)]
  \item The class \(\mathcal{S}_{m}^{(q_1, q_2)}\) arises from the \(L^1\)-scaling invariance of \eqref{PME}, see $(1.3)$ of \cite[Definition~1.1]{HKK01}, when \(m > 1 - \tfrac{1}{d}\), under the scaling structure
    \[
    u_{r}(x,t) = r^{d} \, u\left(rx, r^{2 + d(m-1)}t\right),
    \quad
    V_{r}(x,t) = r^{1 + d(m-1)} \, V\left(rx, r^{2 + d(m-1)}t\right).
    \]
  \item The \(\sigma\)-class \(\mathcal{S}_{m,\sigma}^{(q_1,q_2)}\) is similarly tied to the \(L^1\)-scaling invariance of \eqref{PME} as in $(1.4)$ of \cite[Definition~1.1]{HKK01}. Moreover, for the classes \(\mathcal{S}_{m,\sigma}^{(q_1,q_2)}\) and \(\mathfrak{S}_{m,\sigma}^{(q_1,q_2)}\), one needs \(m > \left(1-\tfrac{2}{d}\right)_+\) to ensure \(2 + d(m-1) > 0\).
%  \item {\color{dred} The class $\mathcal{S}_{m,\sigma}^{(q_1,q_2)}$ is the admissible $\sigma$-critical exponent class used in the divergence-free setting. It lies above the standard $L^1$-scaling critical line $\mathcal{S}_{m}^{(q_1,q_2)}$ in the $(1/q_1,1/q_2)$-plane. The condition $m>\left(1-\frac2d\right)_+$ guarantees that $2+d(m-1)>0$.}
\end{enumerate}
\end{remark}

First, we establish existence results for $1 - \frac{1}{d} < m \leq 2$.
We state the following theorem separately for the porous medium case $1\leq m \leq 2$, and the fast diffusion case $1-\frac{1}{d}< m <1$, respectively, assuming the appropriate functional classes for $V$.

\begin{theorem}\label{T:PME}
Let $m>0$, $d\geq 2$, and $V:\Omega_{T} \to \mathbb{R}^d$ be a measurable vector field. We recall that $\mathcal{S}_{m}^{(q_1,q_2)}$ and $\mathfrak{S}_{m}^{(q_1,q_2)}$ are defined in \eqref{S} and \eqref{sub_S}, respectively.
\begin{enumerate}[(i)] 
  \item (Porous medium case) If $1\leq m < 2$, assume that either
\begin{equation}\label{T:V:PME}
\begin{gathered}
	V \in \mathfrak{S}_{m}^{(q_1, q_2)} \ \text{ for } \
		 q_1 > \frac{md}{(2-m)+d(m-1)} \ \text{ and } \ q_2 \geq 2,	\\
\text{or } \quad V \in \mathcal{S}_{m}^{(q_1,q_2)} \ \text{ for } \ 
\begin{cases}
   (q_1, q_2) = (\frac{2}{m-1}, 2), \ & \text{ if }\ 1<m < 2 \\
  (q_1, q_2) = (\infty, 2), & \text{ if } \ m=1.
\end{cases}
% q_1 = \frac{2}{m-1} \ \text{ and } \ q_2 = 2 .
\end{gathered}
\end{equation}

If $m=2$, assume that either
\begin{equation}\label{T:V:PME_2}
%\begin{gathered}
	V \in \mathfrak{S}_{2}^{(q_1, q_2)} \ \text{ for } \
		 q_1 \geq 2 \ \text{ and } \ q_2 \geq 2,	
\ \text{ or } \  V \in \mathcal{S}_{2}^{(2,2)}.
% q_1 = \frac{2}{m-1} \ \text{ and } \ q_2 = 2 .
%\end{gathered}
\end{equation}
Then for any $\alpha \in (0,2)$, there exists a nonnegative weak solution $u$ of \eqref{PME} such that
\begin{equation}\label{T:Energy}
 \int_{\Omega_{T}} \left|\nabla u^{\frac{m}{2}}\right|^{\alpha} \,dxdt\leq  C\left(\alpha, m, d, |\Omega_T|, \|V\|_{L_{x,t}^{q_1,q_2}}, \mu(\Omega_{T})\right).
\end{equation}
  \item (Fast diffusion case) Let $1-\frac{1}{d} < m <1$. Assume that
\begin{equation}\label{T:V:FDE}
	V \in \mathfrak{S}_{m}^{(q_1, q_2)} \ \text{ for } \ q_1 > \frac{d}{1+d(m-1)} \ \text{ and } \ q_2>\frac{2+d(m-1)}{1+d(m-1)}.
\end{equation}
Then for any $\alpha \in (0,2)$, there exists a nonnegative weak solution $u$ of \eqref{PME} such that \eqref{T:Energy} holds.
\end{enumerate}
\end{theorem}

%%%%%%%%%%%%%%%%%%%%%%%%%
\begin{figure}
\centering

\begin{tikzpicture}[domain=0:16]

% Region
\fill[fill= gray]
(0, 4)--(0.9, 4)-- (3.5,1.1) -- (3.5, 0)--(0, 0) ;

%\fill[fill= gray]
%(2.15, 2.5)-- (4.5, 0.97) -- (3.4, 0) -- (2.65, 0) ;

\draw[->] (0,0) node[left] {\scriptsize $O$} -- (6,0) node[right] {\scriptsize $\frac{1}{q_1}$};
\draw[->] (0,0) -- (0,6) node[left] { \scriptsize $\frac{1}{q_2}$};

% S_{m,1}
\draw[thick, dotted] (0,5) node{\scriptsize $\times$} node[left] {\scriptsize $\frac{1}{2}+\frac{d(m-1)}{2[2+d(m-1)]}$} -- (4.5, 0) node {\scriptsize $\times$} node[below] {\scriptsize $\frac{1+d(m-1)}{d}$} ;
\draw (2.4, 2.4) node[right] {\scriptsize $\mathcal{S}_{m}^{(q_1, q_2)}$} ;
\draw[thick] (3.5,1.1)  circle(0.05) node[right]{\scriptsize $B$};
\draw[thick] (3.5,0)  circle(0.05) node[below]{\scriptsize $B'$};
\draw[dotted] (3.5,1.1) -- (3.5, 0);
\draw[dotted] (3.5,1.1) -- (0, 1.1) node{\scriptsize $\times$} node[left]{\scriptsize $\frac{m-1}{m}$};

\draw (0,4) -- (2.8,0) node{\scriptsize $+$} node[below]{\scriptsize $\frac{1}{d}$};
\draw (2.4, 0.6) node[right] {\scriptsize $\mathcal{S}_{1}^{(q_1, q_2)}$} ;

% Intersection point
\draw[thick] (0.9,4) node{\scriptsize $\bullet$} node[right] {\scriptsize $A=(\frac{m-1}{2}, \frac{1}{2})$};
\draw[thick] (0,4) node{\scriptsize $\bullet$} node[left] {\scriptsize $A'$};
\draw[thick] (0.9,0) node{\scriptsize $\bullet$} node[below] {\scriptsize $A''$};
\draw[dotted] (0.9, 0) -- (0.9, 4);
\draw (0, 4) -- (0.9, 4);

% Notes
\draw (5,5) node[right]{\scriptsize $\bullet$ The shaded region $\mathcal{R}(OA'ABB')$, excluding segments $\overline{AB}$ and $\overline{BB'}$,};
\draw (5.5,4.5) node[right]{\scriptsize belongs to $\mathfrak{S}_{m}^{(q_1,q_2)}$ and satisfies \eqref{T:V:PME}.};
\draw (5,4) node[right]{\scriptsize $\bullet$ Point $A$ on $\mathcal{S}_{m}^{(q_1, q_2)}$ satisfies \eqref{T:V:PME}.};

\draw (5,3) node[right]{\scriptsize $\bullet$ On $\mathcal{S}_{m}^{(q_1, q_2)}$: $\frac{d}{q_1} + \frac{2+d(m-1)}{q_2}=1+d(m-1)$,} ;
\draw (5.5,2.5) node[right]{\scriptsize where $A = (\frac{m-1}{2}, \frac{1}{2})$ and $B= (\frac{(2-m)+d(m-1)}{md}, \frac{m-1}{m})$.} ;
\draw (5,2) node[right]{\scriptsize $\bullet$ When $m=2$, $A=B$; thus $\mathcal{R}(OA'AA'')$ is the valid region for \eqref{T:V:PME_2}.} ;

\draw (5,1.5) node[right]{\scriptsize $\bullet$ On $\mathcal{S}_{1}^{(q_1, q_2)}$: $\frac{d}{q_1} + \frac{2}{q_2}=1$ and $A' = (0, \frac{1}{2})$.} ;
\draw (5,1) node[right]{\scriptsize $\bullet$ In the limit $m\to 1$, $B' \to (\frac{1}{d}, 0)$ and $A \to A'$.} ;

\end{tikzpicture}
\caption{\footnotesize  The condition on $V$ in Theorem~\ref{T:PME} (i) when $1\leq m \leq 2$.}
\label{F:PME01}
\end{figure}
%%%%%%%%%%%%%%%%%%%%%%%%%%%%%%

%%%%%%%%%%%%%%%%%%%%%%%%%%%%%%%%%%%%
\begin{figure}
\centering

\begin{tikzpicture}[domain=0:16]

\fill[fill= gray](0,3.6)-- (0,0) -- (2.5,0);
%\fill[fill= gray](0,3.6)-- (3.4,0) -- (0.8,0) -- (0,1.5);

\draw[->] (0,0)node[left] {\scriptsize $O$} -- (5.5,0) node[right] {\scriptsize $\frac{1}{q_1}$};
\draw[->] (0,0) -- (0,5.5) node[left] { \scriptsize $\frac{1}{q_2}$};

%\draw[violet] (0,0)  node{\scriptsize $\bullet$} node[left] {\scriptsize $\alpha = \frac{2d(2-m)}{2+md}$};

\draw[gray] (0,5) node{\scriptsize $+$} node[left] {\scriptsize $\frac{1}{2}$}
-- (4.5, 0) node {\scriptsize $+$} node[below] {\scriptsize $\frac{1}{d}$} ;
\draw[gray] (4.2, 1) node {\scriptsize $\mathcal{S}_{1}^{(q_1,q_2)}$} ;

\draw[thick, dotted] (0,3.6)--(2.5, 0) ;
 \draw[thick] (0,3.6)  circle(0.05) node[left]{\scriptsize $A$};
 \draw[thick] (2.5,0)  circle(0.05) node[below]{\scriptsize $B$};
\draw (1.8, 1) node[right]{\scriptsize $\mathcal{S}_{m}^{(q_1, q_2)}$};

% Notes

\draw(5,5) node[right]{\scriptsize $\bullet$ The shaded region $\mathcal{R}(OAB)$, excluding the segment $\overline{AB}$, };
\draw(5.5,4.5) node[right]{\scriptsize belongs to $\mathfrak{S}_{m}^{(q_1,q_2)}$ and satisfies \eqref{T:V:FDE}. };

\draw(5,3.5) node[right]{\scriptsize $\bullet$ On $\mathcal{S}_{m}^{(q_1, q_2)}$: $\frac{d}{q_1} + \frac{2+d(m-1)}{q_2}=1+d(m-1),$};
\draw(5.5,3) node[right]{\scriptsize where $A=(0,\frac{1}{2}+\frac{d(m-1)}{2[2+d(m-1)]})$ and $B=(\frac{1+d(m-1)}{d},0)$.};
\draw(5,2.5) node[right]{\scriptsize $\bullet$ In the limit as $m \to 1$, $A \to (0,\frac{1}{2})$ and $B \to (\frac{1}{d},0)$. };
\draw(5,2) node[right]{\scriptsize $\bullet$ In the limit as $m \to 1-\frac{1}{d}$, both $A$ and $B$ approach the origin. };

\end{tikzpicture}
\caption{\footnotesize  The condition on V in Theorem~\ref{T:PME} (ii) when $1-\frac{1}{d} < m < 1$.}
\label{F:FDE01}
\end{figure}
%%%%%%%%%%%%%%%%%%%%%%%%%%%%%%%%%%%%

Here are a few remarks concerning the above theorem.

%{\color{violet} 
%\begin{remark}\label{m=1}
%As stated in Theorem~\ref{T:PME}, the conclusion also holds in the linear diffusion case ($m=1$). Although this fact is likely known to experts, we have not found an explicit reference in the literature. In the special case $\mu=0$, it is well known that weak solutions exist whenever $V\in\mathcal{S}_{1}^{(q_{1},q_{2})}$ defined by \eqref{S} with $m=1$. In the presence of measure data, our result covers the entire subcritical class $\mathfrak{S}_{1}^{(q_{1},q_{2})}$ defined by \eqref{sub_S}, as well as the endpoint $(q_{1},q_{2})=(\infty,2)$ on $\mathcal{S}_{1}^{(q_1,q_2)}$. Therefore, the missing cases consist of the rest of the critical class $\mathcal{S}_{1}^{(q_{1},q_{2})}$, including the opposite endpoint $(q_{1},q_{2})=(d,\infty)$. We conjecture that our restriction is optimal, but as it has not been rigorously verified, we leave the remaining critical cases as an open problem.
%\end{remark}
%}

\begin{remark}\label{m=1}
Theorem~\ref{T:PME} also holds for the linear diffusion case ($m=1$). While this is likely known to experts, we found no explicit reference. For $\mu=0$, weak solutions are known to exist whenever $V\in\mathcal{S}_{1}^{(q_{1},q_{2})}$. In the presence of measure data, our result covers the entire subcritical class $\mathfrak{S}_{1}^{(q_{1},q_{2})}$ and the single endpoint $(q_{1},q_{2})=(\infty,2)$ of the critical line $\mathcal{S}_{1}^{(q_1,q_2)}$. The rest of $\mathcal{S}_{1}^{(q_{1},q_{2})}$, including the opposite endpoint $(d,\infty)$, remains uncovered. We conjecture this restriction is optimal, leaving the remaining critical cases as an open problem.
\end{remark}

\begin{remark}\label{R:T:PME}
\begin{enumerate}[(i)]
    \item
    In \cite{HKK02,HKK_FDE}, the equation \eqref{PME_drift} (without measure data) was investigated on a bounded domain \(\Omega_{T}\) subject to no-flux boundary conditions.
    When the initial data \(\rho_0\) satisfies $\int_{\Omega} \rho_0 \log \rho_0 \, dx < \infty$,
    \(L^1\)-weak solutions were constructed under similar conditions on \(V\) (see \cite[Theorem~2.4]{HKK02} for \(1 \le m \le 2\), and \cite[Theorem~2.5]{HKK_FDE} for \(1 - \tfrac{1}{d} < m < 1\)).
    Unlike the strict inequalities required by \eqref{T:V:PME} and \eqref{T:V:FDE}, those earlier results allow equalities in both the structural class and the ranges of \(q_1\) and \(q_2\).
    
    \item
    The structure of \(V\) is illustrated in Fig.~\ref{F:PME01} for the porous medium case and Fig.~\ref{F:FDE01} for the fast diffusion case.
\end{enumerate}
\end{remark}

\begin{remark} 
The estimates in Theorem~\ref{T:PME} can be generalized as follows.
	Let $1\leq q \le \min\{m+1, 3-m\}$ and $d\geq 2$, suppose that $V$ satisfies either \eqref{T:V:PME} for $1\leq m \leq 2$ or \eqref{T:V:FDE} for $1-\frac{1}{d}< m < 1$. Then for any
	$0 < \alpha < \frac{2(2+md)}{2+d(m+q-1)}$, the following estimate holds
    \begin{equation}\label{Estimate05}
\int_{\Omega_T} \left|\nabla u^{\frac{m+q-1}{2}}\right|^{\alpha} \,dxdt
\leq  C\left(\alpha, m, d, q,  |\Omega_T|, \|V\|_{L_{x,t}^{q_1,q_2}}, \mu(\Omega_{T})\right).
	\end{equation}	
The above estimate can be obtained by following a similar argument as in the proof of Theorem~\ref{T:final_est}, based on \eqref{Estimate03} in Lemma~\ref{comparison estimates}.  When $q=1$, the result coincides with Theorem~\ref{T:PME}.
\end{remark}

%%%%%%%%%%%%%%%%%%%%%%%%%%%%%%%%%%%%%%%%%%%%%%%%%%%%%%%%%%%%%%%%%%%%%%%%%%%%%%%%%%%%
In the divergence-free case, existence results can be obtained under significantly weaker assumptions on both \(V\) and \(m\).
This is because a priori estimates in Lemma~\ref{L:divfree:apriori}, which are essential for determining the suitable functional classes for the solution, are derived independently of \(V\).
However, due to the compactness argument in Lemma~\ref{time derivative}, an additional restriction on $V \in \mathcal{S}_{m,\sigma}^{(q_1,q_2)}$ is required, which belongs to the supercritical region of \(\mathcal{S}_{m}^{(q_1,q_2)}\).
Moreover, the following theorem extends the range of admissible \(m\) to $m > (1 - \frac{2}{d})_{+}$.

\begin{theorem}\label{T:divfree} (Divergence-free case) Let $m>0$, $d\geq 2$, and $V:\Omega_{T} \to \mathbb{R}^d$ be a measurable vector field with  $\nabla \cdot V =  0$ in the sense of distribution. We recall that $\mathcal{S}_{m,\sigma}^{(q_1,q_2)}$ and $\mathfrak{S}_{m,\sigma}^{(q_1,q_2)}$ are defined in \eqref{S_divfree} and \eqref{sub_S_divfree}, respectively.

\begin{enumerate}[(i)]
  \item (Porous medium case) Let $m \geq 1$. Furthermore, assume either
\begin{equation}\label{T:V:divfree:PME}
\begin{gathered}
	V\in \mathfrak{S}_{m,\sigma}^{(q_1, q_2)}\  \text{ for } \
	q_1 > \frac{md}{2+d(m-1)} \ \text{ and } \ q_2 > 1, \\
	\text{or } \quad V\in \mathcal{S}_{m,\sigma}^{(q_1, q_2)}\  \text{ for } \
	(q_1, q_2) = (\infty,1).
\end{gathered}
\end{equation}
Then for any $\alpha \in (0,2)$, there exists a nonnegative weak solution of \eqref{PME} such that
\begin{equation}\label{T:Energy:divfree}
\int_{\Omega_{T}} \left|\nabla u^{\frac{m}{2}}\right|^{\alpha} \,dxdt \leq C\left(\alpha, m, d, |\Omega_T|, \mu(\Omega_{T})\right).
\end{equation}
  \item (Fast diffusion case) Let $(1-\frac{2}{d})_{+} < m < 1$.  Furthermore, assume either
\begin{equation}\label{T:V:divfree:FDE}
\begin{gathered}
	V\in \mathfrak{S}_{m,\sigma}^{(q_1, q_2)}
	\ \text{ for } \ q_1 > \frac{d}{2+d(m-1)} \ \text{ and } \ q_2>1,\\
	\text{or } \quad V\in \mathcal{S}_{m,\sigma}^{(q_1, q_2)}\  \text{ for } \
	(q_1, q_2) = (\infty,1).	
\end{gathered}
\end{equation}
Then for any $\alpha \in (0,2)$, there exists a nonnegative weak solution of \eqref{PME} such that \eqref{T:Energy:divfree} holds.
\end{enumerate}	
\end{theorem}

%%%%%%%%%%%%%%%%%%%%%%%%%
\begin{figure}
\centering

\begin{tikzpicture}[domain=0:16]

% Region

\fill[fill= lgray] (0, 4.5)-- (0,0) --(4,0)-- (4,1.25) --(0, 4.5) ;

\fill[fill= gray] (0,4.5) --(0,0) --(1.5,0) ;

\draw[->] (0,0) node[left] {\scriptsize $O$} -- (7,0) node[right] {\scriptsize $\frac{1}{q_1}$};
\draw[->] (0,0) -- (0,6) node[left] { \scriptsize $\frac{1}{q_2}$};

% S_{m,1}
\draw[thick, dotted] (0,4.5) node{\scriptsize $\bullet$} node[left] {\scriptsize $A$} -- (5.5, 0) ;
\draw (1.5, 3.3) node[right] {\scriptsize $\mathcal{S}_{m>1,\sigma}^{(q_1, q_2)}$} ;

\draw[thick] (4,1.25)  circle(0.05) node[right]{\scriptsize $B$};
\draw[thick] (4,0)  circle(0.05) node[below]{\scriptsize $B'$};
\draw[dotted] (4, 1.25) -- (4, 0);
\draw[dotted] (4, 1.25) -- (0, 1.25) node{\scriptsize $\times$} node[left]{\scriptsize $\frac{m-1}{m}$};

\draw[thick] (5.5,0)  circle(0.05) node[below]{\scriptsize $\frac{2+d(m-1)}{d}$};

% \mathcal{S}_{m, \sigma}^{(q_1, q_2)} 
\draw[very thin] (0, 4.5)  -- (3.2, 0) node{\scriptsize $\times$} node[below]{\scriptsize $\frac{2}{d}$};
\draw[dgray] (1.8, 2) node[right]{\scriptsize $\mathcal{S}_{1, \sigma}^{(q_1, q_2)}$};

% S_{m,1}, m<1
\draw[thick, dotted] (0,4.5) -- (1.5, 0) ;
\draw[thick] (1.5,0)  circle(0.05) node[below]{\scriptsize $C$};
\draw (0.2, 0.3) node[right] {\scriptsize $\mathcal{S}_{m<1,\sigma}^{(q_1, q_2)}$} ;

% \mathcal{S}_{m}^{(q_1, q_2)} 
\draw[dashed, very thick] (0, 3) node{$\ast$}node[left]{\scriptsize $\frac{1+d(m-1)}{2+d(m-1)}$} -- (2.4, 0)  node{$\ast$} node[below]{\scriptsize $\frac{1+d(m-1)}{d}$};
\draw (1.6, 1) node[right]{\scriptsize $\mathcal{S}_{m>1}^{(q_1, q_2)}$};

% Notes

\draw (5,6.5) node[right]{\scriptsize $\star$ Porous medium case ($m \geq 1$): } ;
\draw (5.5,6) node[right]{\scriptsize $\bullet$ The line $\overline{AB}$ is $\mathcal{S}_{m,\sigma}^{(q_1,q_2)}$: $\frac{d}{q_1} + \frac{2+d(m-1)}{q_2} = 2+d(m-1)$, } ;
\draw (6,5.5) node[right]{\scriptsize where $A=(0,1)$, $B=(\frac{2+d(m-1)}{md},\frac{m-1}{m})$, and $B'=(\frac{2+d(m-1)}{md},0)$. } ;
\draw (5.5,5) node[right]{\scriptsize $\bullet$ The region $\mathcal{R}(OABB')$, excluding segments $\overline{AB}$ and $\overline{BB'}$,} ;
\draw (6,4.5) node[right]{\scriptsize belongs to $\mathfrak{S}_{m,\sigma}^{(q_1,q_2)}$ and satisfies \eqref{T:V:divfree:PME}.} ;
\draw (5.5, 4) node[right]{\scriptsize $\bullet$ The point $A=(0,1)$ on $\mathcal{S}_{m,\sigma}^{(q_1, q_2)}$ satisfies \eqref{T:V:divfree:PME}.};

\draw (5,3.5) node[right]{\scriptsize $\star$ Fast diffusion case ($(1-\frac{2}{d})_{+} < m < 1$): } ;
\draw (5.5,3) node[right]{\scriptsize $\bullet$ The line $\overline{AC}$ is $\mathcal{S}_{m,\sigma}^{(q_1,q_2)}$ where $A=(0,1)$ and $C= (\frac{2+d(m-1)}{d}, 0)$.} ;
%\draw (6,2) node[right]{\scriptsize where $A=(0,1)$ and $C= (\frac{2+d(m-1)}{d}, 0)$. } ;
\draw (5.5,2.5) node[right]{\scriptsize $\bullet$ The region $\mathcal{R}(OAC)$, excluding the segment $\overline{AC}$,} ;
\draw (6,2) node[right]{\scriptsize belongs to $\mathfrak{S}_{m,\sigma}^{(q_1,q_2)}$ and satisfies \eqref{T:V:divfree:FDE}.} ;
\draw (5.5,1.5) node[right]{\scriptsize $\bullet$ The point $A=(0,1)$ on $\mathcal{S}_{m,\sigma}^{(q_1, q_2)}$ satisfies \eqref{T:V:divfree:FDE}.};

% Remark about S_m inclusion
\draw (5,1) node[right]{\scriptsize $\bullet$ \textbf{Remark:} The standard critical line $\mathcal{S}_{m}^{(q_1,q_2)}$ (blue, dashed) strictly } ;
\draw (5.5,0.5) node[right]{\scriptsize lies inside the divergence-free subcritical region $\mathfrak{S}_{m,\sigma}^{(q_1,q_2)}$. } ;

\end{tikzpicture}
\caption{\footnotesize  The condition of $V$ in Theorem~\ref{T:divfree} in case $\nabla \cdot V = 0$ when $m> (1-\frac{2}{d})_{+}$.}
\label{F:divfree}
\end{figure}
%%%%%%%%%%%%%%%%%%%%%%%%%%%%%%

Here, we state a few remarks concerning Theorem~\ref{T:divfree}:
\begin{remark}\label{R:T:divfree}
\begin{enumerate}[(i)]
	\item In the divergence-free case, the energy estimate in Theorem~\ref{T:final_est}(iii) is obtained independently of $V$. The conditions on $V$ in \eqref{T:V:divfree:PME} and \eqref{T:V:divfree:FDE} follow directly from \eqref{V:uV01} and \eqref{V:uV02} in Lemma~\ref{time derivative}, which are necessary to guarantee that $uV \in L^{1}(\Omega_T)$, as required by the definition of weak solutions in Definition~\ref{D:WS}.  
	
    \item In \cite{HKK02,HKK_FDE}, the authors studied equation \eqref{PME_drift} without the measure-type forcing term in the bounded domain $\Omega_{T}$. In the divergence-free case, when the initial data $\rho_0$ satisfies $\int_{\Omega} \rho_0 \log \rho_0 \,dx < \infty$, one can construct $L^1$-weak solutions under conditions on $V$ that allow for equality in both the class and the range of $(q_1, q_2)$, compared to \eqref{T:V:divfree:PME} for $m \ge 1$ and \eqref{T:V:divfree:FDE} for $(1-\frac{2}{d})_{+} < m < 1$.
	 
     \item In Fig.~\ref{F:divfree}, the admissible region for $V$ satisfying \eqref{T:V:divfree:PME} is shown as the dark-shaded region $\mathcal{R}(OABB')$, while the region satisfying \eqref{T:V:divfree:FDE} is illustrated as the shaded region $\mathcal{R}(OAC)$. The figure also indicates that the class $\mathcal{S}_{m,\sigma}^{(q_1,q_2)}$ strictly lies above $\mathcal{S}_{m}^{(q_1,q_2)}$.
         
%     \item {\color{dred} In Fig.~\ref{F:divfree}, the admissible region for $V$ satisfying \eqref{T:V:divfree:PME} is shown as the dark-shaded region $\mathcal{R}(OABB')$, while the region satisfying \eqref{T:V:divfree:FDE} is illustrated as the shaded region $\mathcal{R}(OAC)$. The figure also indicates that the standard critical line $\mathcal{S}_{m}^{(q_1,q_2)}$ lies strictly inside the $\sigma$-subcritical region $\mathfrak{S}_{m,\sigma}^{(q_1,q_2)}$. In particular, $\mathcal{S}_{m,\sigma}^{(q_1,q_2)}$ should be understood as a critical line lying above $\mathcal{S}_{m}^{(q_1,q_2)}$, not as a set containing $\mathcal{S}_{m}^{(q_1,q_2)}$.
%         }
\end{enumerate}
\end{remark}

%%%%%%%%%%%%%%%%%%%%%%%%%%%%%%%%%%%%%%%%%%%%%%%%%%%%%%%%%%%%

Let us make a remark for different type of initial data.
\begin{remark}\label{L0 rmk}
When the initial datum is a nonnegative finite measure $\mu_0$ on $\Omega$, that is, $u(\cdot,0)=\mu_0$ in the sense of measures, the corresponding results remain valid with $\mu(\Omega_T)$ replaced by $\mu(\Omega_T)+\mu_0(\Omega)$ in \eqref{T:Energy}, \eqref{Estimate05}, and \eqref{T:Energy:divfree}; see Remarks~\ref{L1 rmk} and \ref{L2 rmk}.
\end{remark}

We remark comparing our results with those obtained in the absence of the drift term, $V=0$.
\begin{remark}\label{R:divfree-q-est}
The gradient estimate in Theorem~\ref{T:divfree} can be generalized as follows. Let $d\geq 2$, $m>\left(1-\frac{2}{d}\right)_+$, and $1<q\le m+1$. Suppose that $V$ satisfies either \eqref{T:V:divfree:PME} for $m\geq 1$ or \eqref{T:V:divfree:FDE} for $\left(1-\frac{2}{d}\right)_+<m<1$. Then for any $0<\alpha<
\frac{2(2+md)}{2+d(m+q-1)}$, the weak solution constructed in Theorem~\ref{T:divfree} satisfies
\begin{equation*}\label{Estimate05:divfree}
\int_{\Omega_T}
\left|\nabla u^{\frac{m+q-1}{2}}\right|^\alpha\,dxdt
\le
C\left(\alpha,m,d,q,|\Omega_T|,\mu(\Omega_T)\right),
\end{equation*}
where the constant $C$ is independent of $V$.
In particular, when $V=0$, the choice $q=2$
for $m>1$ recovers the corresponding estimate of \cite{BDG13}, while the choice $q=m+1$ for $\left(1-\frac{2}{d}\right)_+<m<1$ recovers the corresponding estimate of \cite{BDG15}.
\end{remark}

%---------------------------------------

\subsection{Preliminaries}\label{S:Preliminaries}

In this subsection, we collect several known results that will be used repeatedly throughout the paper.
We begin with a parabolic embedding theorem.

\begin{lemma}[See {\cite[Propositions~I.3.1 \& I.3.2]{DB93}}]\label{T:pSobolev}
Let $v \in L^{\infty}(0,T;L^{q}(\Omega)) \cap L^{p}(0,T;W_0^{1,p}(\Omega))$ for some $1 \leq p < \infty$, and let $0 < q < \infty$.
Then there exists a constant $c=c(d,p,q)\ge1$ such that
\[
\int_{\Omega_T} |v(x,t)|^{\frac{p(d+q)}{d}}\,dxdt
\leq
c \left(\sup_{t\in(0,T)}\int_{\Omega} |v(x,t)|^{q}\,dx \right)^{\frac{p}{d}}
\int_{\Omega_T} |\nabla v(x,t)|^{p}\,dxdt.
\]
\end{lemma}

The following interpolation lemma is similar to \cite[Lemma~3.4]{HKK02} and \cite[Proposition~3.3]{DB93}.

\begin{lemma}\label{L:V}
	Let $m>0$, $d\geq 2$, and $\alpha \in \left[ \max\left\{1, \frac{2d}{2+md}\right\}, \, 2 \right)$. Suppose that $v\geq 0$ satisfies 
	\[
	v \in L^{\infty}(0, T; L^1 (\Omega)) \quad \text{ and } \quad
	v^{\frac{m}{2}} \in L^{\alpha} (0, T; W_{0}^{1, \alpha}(\Omega)).
	\]
	Then, 
	\[
	v \in L^{r_1, r_2}_{x,t}(\Omega_{T})
	\]
	for any pair $(r_1, r_2)$ satisfying
	\begin{equation}\label{Lr1r2}
	\begin{aligned}
	\frac{d}{r_1} + \frac{ \alpha (2+ md) - 2d}{2r_2} = d, \quad 
	\text{ for } \ 1 \leq r_1 \leq \frac{\alpha md}{2(d-\alpha)}, \quad  \max\left\{1, \frac{\alpha m}{2}\right\}\leq r_2 \leq \infty.
	\end{aligned}
	\end{equation}
	Furthermore, we have the estimate
	\begin{equation}\label{Lr1r2_norm}
	\|v\|_{L^{r_1, r_2}_{x,t}} \leq c\left(\sup_{t\in(0,T)} \int_{\Omega} v \,dx\right)^{1-\frac{\alpha m}{2r_2}} \left(\int_{\Omega_T} \left|\nabla v^{\frac{m}{2}}\right|^{\alpha}\,dxdt\right)^{\frac{1}{r_2}},
	\end{equation}
	where $c=c(\alpha, m, d)>0$ is a constant.
\end{lemma}

\begin{proof}
Since $\alpha<d$, the Sobolev embedding theorem yields
\begin{equation*}\label{Sobolev}
\|v^{\frac{m}{2}}\|_{L_{x}^{\frac{\alpha d}{d-\alpha}}} \leq c \|\nabla v^{\frac{m}{2}}\|_{L_{x}^{\alpha}}.
\end{equation*}
Let $0 \leq \theta \leq 1$, and define $r_1$ by
	\begin{equation}\label{L:V00}
	\frac{1}{r_1} = (1-\theta) + \frac{2 \theta  (d-\alpha)}{\alpha md}.
	\end{equation}
By interpolation between $L_{x}^{1}$ and $L_{x}^{\frac{\alpha m d}{2(d-\alpha)}}$, we obtain 
\begin{equation}\label{L:V01}
			\|v\|_{L^{r_1}_{x}} \leq
			\|v\|_{L^1_x}^{1-\theta} \|v\|^{\theta}_{L_{x}^{\frac{\alpha md}{2(d-\alpha)}}}
			= \|v\|_{L^1_x}^{1-\theta} \|v^{\frac{m}{2}}\|^{\frac{2\theta}{m}}_{L_{x}^{\frac{\alpha d}{d-\alpha}}}
			\leq c \|v\|_{L^1_x}^{1-\theta} \|\nabla v^{\frac{m}{2}}\|^{\frac{2\theta}{m}}_{L_{x}^{\alpha}}.
\end{equation}
If $\theta =0$, then $r_1 = 1$ and $r_2 = \infty$, then 
\[
\|v\|_{L^{1,\infty}_{x,t}} \leq \sup_{t\in(0,T)} \int_{\Omega} v \,dx.
\]	
Assume that $\theta \in (0, 1]$, and define $r_2$ by
\begin{equation}\label{L:V000}
\frac{1}{r_2} = \frac{2\theta}{\alpha m}.
\end{equation}
Taking the $L_{t}^{r_2}$-norm on both sides of \eqref{L:V01} gives 	
\begin{equation}\label{L:V02}
\|v\|_{L^{r_1,r_2}_{x,t}}
	\leq c\left( \int_{0}^{T} \|v\|_{L^1_x}^{(1-\theta)r_2} \|\nabla v^{\frac{m}{2}}\|^{\frac{2\theta r_2}{m}}_{L_{x}^{\alpha}}
 \,dt\right)^{\frac{1}{r_2}}
 \leq c\left(\sup_{t\in (0, T)}\|v\|_{L^1_x}\right)^{(1-\theta)} \|\nabla v^{\frac{m}{2}}\|_{L_{x,t}^{\alpha}}^{\frac{\alpha}{r_2}}.
	\end{equation}
Since $\theta = \frac{\alpha m}{2 r_2}$, this directly gives \eqref{Lr1r2_norm}. Finally, combining \eqref{L:V00} and \eqref{L:V000} with the condition $\theta \in [0,1]$ yields the relations in \eqref{Lr1r2}.  
\end{proof}

Finally, we recall the Aubin--Lions lemma.

\begin{lemma}[See \cite{Sim87}]\label{AL}
Let $X_0$, $X$, and $X_1$ be Banach spaces such that $X_0$ is compactly embedded in $X$ and $X$ is continuously embedded in $X_1$.
Let $1\leq p,r \leq \infty$.
For $T>0$, define
\[
W = \left\{ v\in L^{p}(0,T;X_0) \,:\, \partial_t v\in L^{r}(0,T;X_1)\right\},
\]
where $\partial_t v$ is understood in the sense of distributions.
If $p<\infty$, then the embedding of $W$ into $L^{p}(0,T;X)$ is compact.
If $p=\infty$ and $r>1$, then the embedding of $W$ into $C([0,T];X)$ is compact.
\end{lemma}

%%%%%%%%%%%%

\section{A priori estimates}\label{S:apriori-est}

In this section, we derive energy-type estimates for regular solutions to \eqref{PME} under the auxiliary assumptions $V \in C^\infty(\Omega_{T})\cap L^\infty(\Omega_T)$ and $\mu \in C^\infty(\Omega_{T})\cap L^\infty(\Omega_T)$ with $\mu \ge0$.
We begin by introducing the notion of a regular solution.
\begin{definition}\label{reg sol}
Let $m>0$. A nonnegative function $u$ is called a regular solution of \eqref{PME} with zero lateral boundary data and zero initial data if
\[
u \in C([0,T];L^{m+1}(\Omega)), \quad u^{m} \in L^2(0,T;W_0^{1,2}(\Omega)),% \quad  u(\cdot,0)=0 \ \ \text{in }L^2(\Omega),
\]
and
\begin{equation}\label{reg sol1}
\int_{\Omega_T} \left\{ - u \varphi_t + \nabla u^m \cdot \nabla \varphi - uV \cdot \nabla \varphi \right\}\,dxdt
=
\int_{\Omega_T} \mu \varphi\,dxdt
\end{equation}
for any $\varphi \in C^{\infty} \left(\overline{\Omega} \times [0, T]\right)$ which vanishes on $\partial \Omega \times (0,T)$ and $\varphi(\cdot,T)=0$.
%\begin{equation*}
%\varphi \in W^{1,\frac{p}{p-1}} \left(0,T; L^{\frac{p}{p-1}}(\Omega)\right) \cap L^r \left(0,T; W^{1,r}_{0}(\Omega)\right)\  \text{with }  \, \varphi(\cdot,T)=0, \quad \text{where }\ r=\frac{2p+d(m+p-1)}{p+d(p-1)} \quad \text{for all } 1\le p \le q.
%\end{equation*}
\end{definition}
%{\color{dgreen} For $\beta > 0$, the symbol $\nabla u^{\beta}$ is understood in the following definition:
%$$
%\nabla u^{\beta} := \frac{2\beta}{m+q-1}\chi_{\{u>0\}}u^{\frac{2\beta-m-p+1}{2}} \nabla u^{\frac{m+p-1}{2}}.
%$$

%Here the symbol $\nabla u^m$ is to be interpreted according to the following definition:
%\begin{equation*}
%\nabla u^m = \frac{2m}{m+q-1}\chi_{\{u>0\}} u^{\frac{m-q+1}{2}} \nabla u^{\frac{m+q-1}{2}}.
%\end{equation*}

\begin{remark}
\begin{enumerate}[(i)]
  \item For $\alpha,\beta>0$, we interpret $\nabla u^\alpha$ by
  \[
  \nabla u^\alpha
  :=
  \frac{\alpha}{\beta}\chi_{\{u>0\}}u^{\alpha-\beta}\nabla u^\beta,
  \]
  whenever the right-hand side is well defined.
  \item Let $1 \le q \le m+1$ and $\delta>0$, and set $u_\delta := \max\{u, \delta\}$. If $\nabla u^m \in L^2(\Omega_T)$, then, by (i),
  \[
  \nabla u_\delta^{\frac{m+q-1}{2}} \in L^2(\Omega_T)
  \qquad\text{and}\qquad
  \nabla u_\delta^{q-1} \in L^2(\Omega_T).
  \]
\end{enumerate}
\end{remark}

%Note that, in the weak formulation \eqref{reg sol2}, we can choose $\varphi = [u]_h^{q-1}$ for $q>1$ as the test function.

We first derive the following estimates by using truncated test functions.
\begin{lemma}\label{L1}
Let $d\geq 2$ and $1-\frac{1}{d} < m \leq 2$.
Suppose that $u$ is a regular solution of \eqref{PME}.
Then it follows that
\begin{equation}\label{Estimate01}
	\sup_{t\in (0, T)} \int_{\Omega} u(x, t) \,dx \leq \mu (\Omega_{T}),
\end{equation}
where $\mu(\Omega_T): = \int_{\Omega_T} \mu(x,t) \, dxdt$. 
Moreover, for any $A>0$, $\xi > 1$, and $1 < q \le m+1$, we have
\begin{equation}\label{Estimate02}
	\int_{\Omega_T} \frac{\left|\nabla u^{\frac{m+q-1}{2}}\right|^2}{\left(A^{q-1} + u^{q-1}\right)^{\xi}} \,dxdt		
	\leq \frac{(m+q-1)^2 A^{(q-1)(1-\xi)}}{m(q-1)(\xi-1)} \left[ 2\mu(\Omega_{T}) + \frac{(q-1)(\xi-1)}{3m}\int_{\Omega_T} u^{2-m} |V|^2 \,dxdt\right].
\end{equation}
\end{lemma}

\begin{proof}
We first note that, since $d\ge2$ and $m>1-\frac1d$, we have
$2-m\le m+1$. Hence, from
$u\in C([0,T];L^{m+1}(\Omega))$ and $V\in L^\infty(\Omega_T)$, it follows that $u^{2-m}|V|^2\in L^1(\Omega_T)$.

We justify the nonlinear test functions by a standard smooth approximation argument. More precisely, we first perform the following computations for smooth non-degenerate approximating solutions $u_\kappa$ with zero lateral boundary data and zero initial data. For such solutions, all chain rules and integrations by parts below are classical, and the nonlinear test functions are admissible. The estimates obtained are independent of the approximation parameter $\kappa$. Hence, passing to the limit $\kappa\to0$, using the strong convergence of the approximations and the lower semicontinuity of the quadratic gradient terms, yields the estimates for the original regular solution. To simplify notation, we omit the subscript $\kappa$ below.

%For the smooth approximating solutions, the weak formulation is equivalently
%written as
%\[
%\int_{\Omega_T}
%\left\{
%u_t\eta+\nabla u^m\cdot\nabla\eta-uV\cdot\nabla\eta
%\right\}\,dxdt
%=
%\int_{\Omega_T}\mu\eta\,dxdt
%\]
%for admissible test functions $\eta$ vanishing on
%$\partial\Omega\times(0,T)$ and satisfying $\eta(\cdot,T)=0$.

%The restriction $q\le m+1$ is used to justify the nonlinear powers appearing in the test functions from the assumption $u^m\in L^2(0,T;W^{1,2}_0(\Omega))$.
%Indeed, on the truncated sets where $u>\delta$, the identities
%\[
%\nabla u^{q-1} = \frac{q-1}{m}u^{q-1-m}\nabla u^m
%\]
%and
%\[
%\nabla u^{\frac{m+q-1}{2}} = \frac{m+q-1}{2m} u^{\frac{q-m-1}{2}}\nabla u^m
%\]
%are controlled by $\nabla u^m$, since $q-1-m\le0$ and $(q-m-1)/2\le0$.

\smallskip
\noindent
\underline{\emph{Step 1: proof of \eqref{Estimate01}.}}

Fix $\varepsilon>0$ and $0<\delta<\varepsilon$. Define
\[
F_{\varepsilon,\delta}(s):=
\begin{cases}
0, & 0\le s\le \delta,\\[1mm]
\dfrac{s^{q-1}-\delta^{q-1}}
{\varepsilon^{q-1}-\delta^{q-1}},
& \delta<s<\varepsilon,\\[3mm]
1, & s\ge \varepsilon,
\end{cases} 
\quad \text{ and }
\quad 
\Psi_{\varepsilon,\delta}(s)
:=
\int_0^sF_{\varepsilon,\delta}(r)\,dr.
\]
Let $\phi\in C^\infty([0,T])$ satisfy $0\le \phi\le1$, $-\phi_t\ge0$, $\phi(T)=0$.
We use
\[
\eta_1:=F_{\varepsilon,\delta}(u)\phi
\]
as a test function. Since $u=0$ on $\partial\Omega\times(0,T)$ and
$F_{\varepsilon,\delta}(0)=0$, we have $\eta_1=0$ on $\partial\Omega\times(0,T)$.
Thus $\eta_1$ is admissible.

Testing the equation by $\eta_1$, we obtain
\[
I+II=III+IV,
\]
where
\[
I:=
\int_{\Omega_T}
u_tF_{\varepsilon,\delta}(u)\phi\,dxdt,
\qquad
II:=
\int_{\Omega_T}
\nabla u^m\cdot\nabla F_{\varepsilon,\delta}(u)\phi\,dxdt,
\]
\[
III:=
\int_{\Omega_T}
\mu F_{\varepsilon,\delta}(u)\phi\,dxdt,
\qquad
IV:=
\int_{\Omega_T}
uV\cdot\nabla F_{\varepsilon,\delta}(u)\phi\,dxdt.
\]

Since $\partial_t\Psi_{\varepsilon,\delta}(u) = u_tF_{\varepsilon,\delta}(u)$, 
integrating by parts in time gives
\[
I
=
\int_{\Omega_T}
\Psi_{\varepsilon,\delta}(u)(-\phi_t)\,dxdt 
-\int_\Omega
\Psi_{\varepsilon,\delta}(u(x,0))\phi(0)\,dx
+
\int_\Omega
\Psi_{\varepsilon,\delta}(u(x,T))\phi(T)\,dx.
\]
The terminal term is zero because $\phi(T)=0$. The initial term is also zero,
because $u(\cdot,0)=0$ and 
$\Psi_{\varepsilon,\delta}(0)=0$.
Therefore
\begin{equation}\label{I01-smooth-u}
I
=
\int_{\Omega_T}
\Psi_{\varepsilon,\delta}(u)(-\phi_t)\,dxdt
\ge0.
\end{equation}

Moreover, since $0\le F_{\varepsilon,\delta}\le1$ and $0\le\phi\le1$,
\begin{equation}\label{III01-smooth-u}
0\le III\le \mu(\Omega_T).
\end{equation}

Next,
\[
\nabla F_{\varepsilon,\delta}(u)
=
\frac{(q-1)u^{q-2}}
{\varepsilon^{q-1}-\delta^{q-1}}
\chi_{\{\delta<u<\varepsilon\}}
\nabla u.
\]
Hence
\[
\begin{aligned}
II
&=
\frac{m(q-1)}
{\varepsilon^{q-1}-\delta^{q-1}}
\int_{\{\delta<u<\varepsilon\}}
u^{m+q-3}
|\nabla u|^2\phi\,dxdt  \\
&=
\frac{4m(q-1)}
{(m+q-1)^2(\varepsilon^{q-1}-\delta^{q-1})}
\int_{\{\delta<u<\varepsilon\}}
\left|\nabla u^{\frac{m+q-1}{2}}\right|^2\phi\,dxdt.
\end{aligned}
\]
In particular, $II\ge0$. Letting $\delta\to0$ and using Fatou's lemma, we get
\begin{equation}\label{II01-smooth-u}
\liminf_{\delta\to0}II
\ge
\frac{4m(q-1)}
{(m+q-1)^2\varepsilon^{q-1}}
\int_{\{0<u<\varepsilon\}}
\left|\nabla u^{\frac{m+q-1}{2}}\right|^2\phi\,dxdt.
\end{equation}

For the drift term, Young's inequality gives
\[
\begin{aligned}
IV
&=
\frac{q-1}{\varepsilon^{q-1}-\delta^{q-1}}
\int_{\{\delta<u<\varepsilon\}}
u^{q-1}V\cdot\nabla u\,\phi\,dxdt  \\
&\le
\frac12 II
+
\frac{q-1}{2m(\varepsilon^{q-1}-\delta^{q-1})}
\int_{\{\delta<u<\varepsilon\}}
u^{q+1-m}|V|^2\phi\,dxdt.
\end{aligned}
\]
On $\{\delta<u<\varepsilon\}$, we have
$
u^{q+1-m}
=
u^{2-m}u^{q-1}
\le
\varepsilon^{q-1}u^{2-m}$.
Therefore
\begin{equation}\label{IV01-smooth-u}
IV
\le
\frac12 II
+
\frac{(q-1)\varepsilon^{q-1}}
{2m(\varepsilon^{q-1}-\delta^{q-1})}
\int_{\{\delta<u<\varepsilon\}}
u^{2-m}|V|^2\phi\,dxdt.
\end{equation}

Combining \eqref{I01-smooth-u}--\eqref{IV01-smooth-u}, and then letting
$\delta\to0$, we obtain
\begin{equation}\label{comp1-smooth-u}
\begin{aligned}
&\int_{\Omega_T}
\left[
\int_0^u
\min\left\{1,\frac{s^{q-1}}{\varepsilon^{q-1}}\right\}\,ds
\right]
(-\phi_t)\,dxdt
+
\frac{2m(q-1)}
{(m+q-1)^2\varepsilon^{q-1}}
\int_{\{0<u<\varepsilon\}}
\left|\nabla u^{\frac{m+q-1}{2}}\right|^2\phi\,dxdt  \\
&\qquad\le
\mu(\Omega_T)
+
\frac{q-1}{2m}
\int_{\{0<u<\varepsilon\}}
u^{2-m}|V|^2\phi\,dxdt.
\end{aligned}
\end{equation}

Letting $\varepsilon\to0$, the last term on the right-hand side tends to
zero because $u^{2-m}|V|^2\in L^1(\Omega_T)$. Also,
\[
\int_0^u
\min\left\{1,\frac{s^{q-1}}{\varepsilon^{q-1}}\right\}\,ds
\longrightarrow u
\qquad\text{in }L^1(\Omega_T).
\]
Thus
\[
\int_{\Omega_T}u(-\phi_t)\,dxdt
\le
\mu(\Omega_T)
\]
for every smooth nonincreasing $\phi$ satisfying
$0\le\phi\le1$ and $\phi(T)=0$.

Now fix $\tau\in(0,T)$. Choose nonnegative functions
$\rho_j\in C_c^\infty(0,T)$ such that
\[
\int_0^T\rho_j(t)\,dt=1,
\qquad
\rho_j\rightharpoonup\delta_\tau
\]
in the sense of measures, and set
\[
\phi_j(t):=\int_t^T\rho_j(s)\,ds.
\]
Then $0\le\phi_j\le1$, $-(\phi_j)_t=\rho_j$, and $\phi_j(T)=0$.
Since $u\in C([0,T];L^{m+1}(\Omega))$, the function
\[
t\mapsto\int_\Omega u(x,t)\,dx
\]
is continuous. Passing to the limit $j\to\infty$ gives
\[
\int_\Omega u(x,\tau)\,dx
\le
\mu(\Omega_T).
\]
Taking the supremum over $\tau\in(0,T)$ proves \eqref{Estimate01}.

\bigskip
\noindent
\underline{\emph{Step 2: proof of \eqref{Estimate02}.}}

Let $A>0$ and $\xi>1$. We again fix $\varepsilon>0$ and
$0<\delta<\varepsilon$, and use the test function
\[
\eta_2:=
\frac{F_{\varepsilon,\delta}(u)\phi}
{\left(A^{q-1}+u^{q-1}\right)^{\xi-1}},
\]
where $\phi\in C^\infty([0,T])$ satisfies $0\le\phi\le1$, $ -\phi_t\ge0$, and $\phi(T)=0$.
Since $F_{\varepsilon,\delta}(0)=0$, the function $\eta_2$ vanishes on
$\partial\Omega\times(0,T)$, and hence it is admissible.

Testing the equation by $\eta_2$ gives
\[
I+II_1+II_2=III+IV_1+IV_2,
\]
where
\[
I:=
\int_{\Omega_T}
u_t
\frac{F_{\varepsilon,\delta}(u)\phi}
{\left(A^{q-1}+u^{q-1}\right)^{\xi-1}}
\,dxdt,
\qquad
II_1:=
\int_{\Omega_T}
\frac{
\nabla u^m\cdot\nabla F_{\varepsilon,\delta}(u)
}
{\left(A^{q-1}+u^{q-1}\right)^{\xi-1}}
\phi\,dxdt,
\]
\[
II_2:=
(1-\xi)
\int_{\Omega_T}
\frac{
F_{\varepsilon,\delta}(u)
\nabla u^m\cdot\nabla u^{q-1}
}
{\left(A^{q-1}+u^{q-1}\right)^\xi}
\phi\,dxdt,
\qquad
III:=
\int_{\Omega_T}
\mu
\frac{F_{\varepsilon,\delta}(u)\phi}
{\left(A^{q-1}+u^{q-1}\right)^{\xi-1}}
\,dxdt,
\]
\[
IV_1:=
\int_{\Omega_T}
\frac{
uV\cdot\nabla F_{\varepsilon,\delta}(u)
}
{\left(A^{q-1}+u^{q-1}\right)^{\xi-1}}
\phi\,dxdt,
\qquad
IV_2:=
(1-\xi)
\int_{\Omega_T}
\frac{
F_{\varepsilon,\delta}(u)
uV\cdot\nabla u^{q-1}
}
{\left(A^{q-1}+u^{q-1}\right)^\xi}
\phi\,dxdt.
\]

Since $F_{\varepsilon,\delta}(u)=0$ on $\{u\le\delta\}$, all terms
involving $\nabla u^{q-1}$ are supported in $\{u>\delta\}$. Hence no
singularity occurs at $u=0$. Moreover,
\[
\nabla u^m\cdot\nabla u^{q-1}
=
m(q-1)u^{m+q-3}|\nabla u|^2
\ge0
\qquad\text{on }\{u>\delta\}.
\]
Since $1-\xi<0$, we have
\[
II_2\le0.
\]
Therefore
\begin{equation}\label{basic-smooth-u}
-II_2
\le
|I|+II_1+III+|IV_1|+|IV_2|.
\end{equation}

Define
\[
\Theta_{\varepsilon,\delta}(s)
:=
\int_0^s
\frac{F_{\varepsilon,\delta}(r)}
{\left(A^{q-1}+r^{q-1}\right)^{\xi-1}}
\,dr.
\]
Then
\[
I
=
\int_{\Omega_T}
\partial_t\Theta_{\varepsilon,\delta}(u)\phi\,dxdt.
\]
As in Step~1, the initial and terminal boundary terms vanish, and hence
\[
I
=
\int_{\Omega_T}
\Theta_{\varepsilon,\delta}(u)(-\phi_t)\,dxdt.
\]
Moreover,
\[
0\le
\Theta_{\varepsilon,\delta}(u)
\le
A^{(q-1)(1-\xi)}u.
\]
Using \eqref{Estimate01}, we get
\begin{equation}\label{I2-smooth-u}
|I|
\le
A^{(q-1)(1-\xi)}\mu(\Omega_T).
\end{equation}
Similarly,
\[
0\le
\frac{F_{\varepsilon,\delta}(u)}
{\left(A^{q-1}+u^{q-1}\right)^{\xi-1}}
\le
A^{(q-1)(1-\xi)},
\]
so
\begin{equation}\label{III2-smooth-u}
0\le III
\le
A^{(q-1)(1-\xi)}\mu(\Omega_T).
\end{equation}

For $II_1$, we have
\[
\begin{aligned}
II_1
&=
\frac{m(q-1)}
{\varepsilon^{q-1}-\delta^{q-1}}
\int_{\{\delta<u<\varepsilon\}}
\frac{
u^{m+q-3}|\nabla u|^2\phi
}
{\left(A^{q-1}+u^{q-1}\right)^{\xi-1}}
\,dxdt \\
&\le
A^{(q-1)(1-\xi)}
\frac{m(q-1)}
{\varepsilon^{q-1}-\delta^{q-1}}
\int_{\{\delta<u<\varepsilon\}}
u^{m+q-3}|\nabla u|^2\phi
\,dxdt.
\end{aligned}
\]
The last integral is precisely the diffusion term from Step~1. Hence, using
\eqref{comp1-smooth-u} before letting $\varepsilon\to0$, we obtain
\begin{equation}\label{II1-smooth-u}
\begin{aligned}
\limsup_{\delta\to0}II_1
&\le
2A^{(q-1)(1-\xi)}
\left[
\mu(\Omega_T)
+
\frac{q-1}{2m}
\int_{\{0<u<\varepsilon\}}
u^{2-m}|V|^2\phi\,dxdt
\right].
\end{aligned}
\end{equation}

Next, by Fatou's lemma,
\begin{equation}\label{II2-smooth-u}
\liminf_{\varepsilon\to0}
\liminf_{\delta\to0}
(-II_2) \ge
\frac{4m(q-1)(\xi-1)}
{(m+q-1)^2}
\int_{\Omega_T}
\frac{
\left|\nabla u^{\frac{m+q-1}{2}}\right|^2\phi
}
{\left(A^{q-1}+u^{q-1}\right)^\xi}
\,dxdt.
\end{equation}

For $IV_1$, Young's inequality gives
\[
\begin{aligned}
|IV_1|
&=
\left|
\frac{q-1}{\varepsilon^{q-1}-\delta^{q-1}}
\int_{\{\delta<u<\varepsilon\}}
\frac{
u^{q-1}V\cdot\nabla u\,\phi
}
{\left(A^{q-1}+u^{q-1}\right)^{\xi-1}}
\,dxdt
\right| \\
&\le
II_1 +
\frac{q-1}{4m(\varepsilon^{q-1}-\delta^{q-1})}
\int_{\{\delta<u<\varepsilon\}}
\frac{
u^{q+1-m}|V|^2\phi
}
{\left(A^{q-1}+u^{q-1}\right)^{\xi-1}}
\,dxdt.
\end{aligned}
\]
Since
$
\left(A^{q-1}+u^{q-1}\right)^{1-\xi}
\le
A^{(q-1)(1-\xi)}
$
and
$
u^{q+1-m}
=
u^{2-m}u^{q-1}
\le
\varepsilon^{q-1}u^{2-m}
$ 
on $\{\delta<u<\varepsilon\}$,
we infer
\begin{equation}\label{IV1-smooth-u}
\limsup_{\delta\to0}|IV_1|
\le
\limsup_{\delta\to0}II_1
+
o_\varepsilon(1),
\end{equation}
where $o_\varepsilon(1)\to0$ as $\varepsilon\to0$.

For $IV_2$, Young's inequality yields
\[
|IV_2|
\le
\frac14(-II_2)
+
\frac{(q-1)(\xi-1)}{m}
\int_{\Omega_T}
\frac{
F_{\varepsilon,\delta}(u)u^{q+1-m}|V|^2\phi
}
{\left(A^{q-1}+u^{q-1}\right)^\xi}
\,dxdt.
\]
Using
$
\left(A^{q-1}+u^{q-1}\right)^\xi
\ge
A^{(q-1)(\xi-1)}u^{q-1}
$ on $\{u>0\}$,
we obtain
\[
\frac{u^{q+1-m}}
{\left(A^{q-1}+u^{q-1}\right)^\xi}
\le
A^{(q-1)(1-\xi)}u^{2-m}.
\]
Therefore
\begin{equation}\label{IV2-smooth-u}
|IV_2|
\le
\frac14(-II_2)
+
\frac{(q-1)(\xi-1)A^{(q-1)(1-\xi)}}{m}
\int_{\Omega_T}u^{2-m}|V|^2\phi\,dxdt.
\end{equation}

Combining
\eqref{basic-smooth-u}--\eqref{IV2-smooth-u}, then letting
$\delta\to0$ and $\varepsilon\to0$, and using
\[
\int_{\{0<u<\varepsilon\}}
u^{2-m}|V|^2\phi\,dxdt
\to0
\qquad\text{as }\varepsilon\to0,
\]
we get
\[
\int_{\Omega_T}
\frac{
\left|\nabla u^{\frac{m+q-1}{2}}\right|^2\phi
}
{\left(A^{q-1}+u^{q-1}\right)^\xi}
\,dxdt \le
\frac{(m+q-1)^2A^{(q-1)(1-\xi)}}
{m(q-1)(\xi-1)}
\left[
2\mu(\Omega_T)
+
\frac{(q-1)(\xi-1)}{3m}
\int_{\Omega_T}u^{2-m}|V|^2\phi\,dxdt
\right].
\]

Finally, choose $\{\phi_k\}_{k=1}^\infty\subset C^\infty([0,T])$ such that
\[
0\le\phi_k\le1,\qquad -(\phi_k)_t\ge0,\qquad \phi_k(T)=0,
\]
\[
\phi_k(t)=1
\quad\text{for }0\le t\le T-\frac1k,
\qquad
\phi_k\nearrow1
\quad\text{pointwise in }(0,T).
\]
Applying the previous estimate with $\phi=\phi_k$ and then letting
$k\to\infty$, the monotone convergence theorem gives \eqref{Estimate02}, which completes the proof.
\end{proof}

\begin{remark}\label{L1 rmk}
Let the initial datum be a finite nonnegative measure $\mu_0$ on $\Omega$. Although such data are not covered directly by Definition~\ref{reg sol}, the estimates in Lemma~\ref{L1} remain valid at the approximation level. More precisely, if one approximates $\mu_0$ by a sequence of nonnegative smooth functions $\{u_{0,\kappa}\}$ such that
\[
u_{0,\kappa} \to \mu_0 \quad \text{in the sense of measures}
\qquad\text{and}\qquad
\int_\Omega u_{0,\kappa}\,dx \to \mu_0(\Omega),
\]
then the corresponding regularized solutions satisfy the same estimates as in Lemma~\ref{L1}, with $\mu(\Omega_T)$ replaced by
\[
\nu(\Omega_T):=\mu(\Omega_T)+\mu_0(\Omega).
\]
Passing to the limit yields the same bounds for solutions with measure-valued initial data.
\end{remark}

In the divergence-free case, we obtain the following analogue of Lemma~\ref{L1}.
\begin{lemma}\label{L:divfree:apriori}
(Divergence-free case)
Let $d\geq 2$ and $m>0$. Assume that $\nabla\cdot V=0$, and suppose that $u$ is a regular solution of \eqref{PME}. Then \eqref{Estimate01} holds. Moreover, for any $A>0$, $\xi>1$, and $1<q\le m+1$, we have
\begin{equation}\label{Estimate02:divfree}
\int_{\Omega_T}
\frac{\left|\nabla u^{\frac{m+q-1}{2}}\right|^2}
{\left(A^{q-1}+u^{q-1}\right)^\xi}
\,dxdt
\le
\frac{(m+q-1)^2A^{(q-1)(1-\xi)}\mu(\Omega_T)}
{2m(q-1)(\xi-1)}.
\end{equation}
\end{lemma}

\begin{proof}
We argue as in the proof of Lemma~\ref{L1}, using the same approximation
argument, notation, and truncated test functions. It remains only to observe
that the drift term vanishes.

Let $\Phi:[0,\infty)\to\mathbb{R}$ be Lipschitz with $\Phi(0)=0$, and let
\[
\eta=\Phi(u)\phi(t),
\]
where $\phi\in C^\infty([0,T])$, $0\le\phi\le1$, $-\phi_t\ge0$, and
$\phi(T)=0$. Define
\[
H_\Phi(s):=\int_0^s r\Phi'(r)\,dr.
\]
Then
\[
\nabla H_\Phi(u)=u\Phi'(u)\nabla u.
\]
Since $u=0$ on $\partial\Omega\times(0,T)$, we have $H_\Phi(u)=0$ on the
lateral boundary. Hence, using $\nabla\cdot V=0$,
\begin{equation}\label{divfree-cancel}
\int_{\Omega_T}uV\cdot\nabla\eta\,dxdt
=
\int_{\Omega_T}\phi V\cdot\nabla H_\Phi(u)\,dxdt
=
-\int_{\Omega_T}\phi(\nabla\cdot V)H_\Phi(u)\,dxdt
=0.
\end{equation}

For \eqref{Estimate01}, we take
\[
\eta_1=F_{\varepsilon,\delta}(u)\phi
\]
as in Lemma~\ref{L1}. By \eqref{divfree-cancel}, the drift term vanishes,
and the identity in Step~1 of Lemma~\ref{L1} reduces to
\[
I+II=III.
\]
The estimates of $I,II,$ and $III$ are the same as in Lemma~\ref{L1}, except
that no drift contribution appears. 
In particular,
\[
I\ge0,\qquad 0\le III\le\mu(\Omega_T),
\]
and hence
\[
I+II\le\mu(\Omega_T).
\]
Letting $\delta\to0$, then $\varepsilon\to0$, and finally approximating a
Dirac mass at any $\tau\in(0,T)$ by $-\phi_t$, exactly as in
Lemma~\ref{L1}, gives
\[
\int_\Omega u(x,\tau)\,dx\le\mu(\Omega_T).
\]
Taking the supremum over $\tau\in(0,T)$ proves \eqref{Estimate01}.

For \eqref{Estimate02:divfree}, we take
\[
\eta_2
=
\frac{F_{\varepsilon,\delta}(u)\phi}
{\left(A^{q-1}+u^{q-1}\right)^{\xi-1}}.
\]
Again, by \eqref{divfree-cancel}, the drift terms vanish. Therefore the
identity in Step~2 of Lemma~\ref{L1} reduces to
\[
I+II_1+II_2=III.
\]
Since $II_2\le0$ and $III\ge0$, we have
\[
-II_2\le I+II_1.
\]
As in Lemma~\ref{L1},
\[
I\le A^{(q-1)(1-\xi)}\mu(\Omega_T).
\]
Moreover,
\[
II_1
\le
A^{(q-1)(1-\xi)}
\int_{\Omega_T}\nabla u^m\cdot\nabla F_{\varepsilon,\delta}(u)\phi\,dxdt.
\]
The last integral is precisely the diffusion term $II$ from Step~1. Since
the divergence-free Step~1 gives $II\le\mu(\Omega_T)$, we obtain
\[
II_1\le A^{(q-1)(1-\xi)}\mu(\Omega_T).
\]
Thus
\[
-II_2\le2A^{(q-1)(1-\xi)}\mu(\Omega_T).
\]
Letting $\delta\to0$ and $\varepsilon\to0$, Fatou's lemma gives, exactly as
in Lemma~\ref{L1},
\[
\frac{4m(q-1)(\xi-1)}
{(m+q-1)^2}
\int_{\Omega_T}
\frac{\left|\nabla u^{\frac{m+q-1}{2}}\right|^2\phi}
{\left(A^{q-1}+u^{q-1}\right)^\xi}
\,dxdt
\le
2A^{(q-1)(1-\xi)}\mu(\Omega_T).
\]
Therefore,
\[
\int_{\Omega_T}
\frac{\left|\nabla u^{\frac{m+q-1}{2}}\right|^2\phi}
{\left(A^{q-1}+u^{q-1}\right)^\xi}
\,dxdt
\le
\frac{(m+q-1)^2A^{(q-1)(1-\xi)}}
{2m(q-1)(\xi-1)}
\mu(\Omega_T).
\]
Finally, taking $\phi_k\nearrow1$ as in Lemma~\ref{L1} and using the
monotone convergence theorem gives the same estimate without $\phi$. Hence \eqref{Estimate02:divfree} follows. Note that since the drift term is canceled, the restriction $1-\frac{1}{d}<m\le2$ used in
Lemma~\ref{L1} is not needed here. Thus the result holds for all $m>0$.
\end{proof}

As a consequence of Lemma~\ref{L1}, we obtain the following estimates.
\begin{lemma}\label{comparison estimates}
Assume the same hypotheses as in Lemma~\ref{L1}. Then the following estimates hold.
\begin{enumerate}[(i)]
  \item Let $1 < q \le m+1$. Also, let $0<\alpha < 2 - \frac{2d(q-1)}{2+d(m+q-1)} = \frac{2(2+md)}{2+d(m+q-1)}$. Then it follows that
	\begin{equation}\label{Estimate03}
\begin{aligned}
&\left(\fint_{\Omega_T} \left|\nabla u^{\frac{m+q-1}{2}}\right|^{\alpha} \,dxdt\right)^{\frac{1}{\alpha}} \\
&\qquad \le c \left[\frac{2d(m+q-1)^2}{m[2(2+md)-\alpha\{2+d(m+q-1)\}]}\right]^{\frac{2+d(m+q-1)}{2(2+md)}} [\mu(\Omega_{T})]^{\frac{q-1}{2+md}} \left[ \frac{2\mu(\Omega_{T})}{|\Omega_T|}\right]^{\frac{2+d(m+q-1)}{2(2+md)}}\\
&\qquad \quad + c [\mu(\Omega_{T})]^{\frac{q-1}{2+md}} \left[\frac{(m+q-1)^2}{m^2}\fint_{\Omega_T} u^{2-m} |V|^2 \,dxdt\right]^{\frac{2+d(m+q-1)}{2(2+md)}},
\end{aligned}
	\end{equation}
for some constant $c=c(m,d,q,\alpha)\ge 1$.
  \item When $q\searrow 1$, the following holds, for any $\alpha \in (0,2)$,
	\begin{equation}\label{Estimate04}
\left(\fint_{\Omega_T} \left|\nabla u^{\frac{m}{2}}\right|^{\alpha} \,dxdt\right)^{\frac{1}{\alpha}}
\le c \left[ \frac{\mu(\Omega_{T})}{|\Omega_T|}\right]^{\frac{1}{2}}+ c \left(\fint_{\Omega_T} u^{2-m} |V|^2 \,dxdt\right)^{\frac{1}{2}}
	\end{equation}
for some constant $c=c(m,d,\alpha)\ge 1$.
\end{enumerate}
\end{lemma}

\begin{remark}
\begin{enumerate}[(i)]
  \item The constant $c$ in \eqref{Estimate03} remains bounded as
  $q\searrow1$.

  \item In the divergence-free case $\nabla\cdot V=0$, under the hypotheses of
  Lemma~\ref{L:divfree:apriori} instead of Lemma~\ref{L1}, the estimates \eqref{Estimate03} and
  \eqref{Estimate04} remain valid with the terms involving $V$ omitted.
\end{enumerate}
\end{remark}

\begin{proof}
The estimate \eqref{Estimate04} follows from the estimate \eqref{Estimate03}. Now we prove \eqref{Estimate03} for
$1\le \alpha < \frac{2(2+md)}{2+d(m+q-1)}$.
The estimate for $\alpha<1$, whenever needed, follows from
H\"{o}lder's inequality after proving the estimate for a larger exponent. First, it follows that by applying the H\"{o}lder inequality,
\begin{equation*}
\begin{aligned}
  \fint_{\Omega_T} \left|\nabla u^{\frac{m+q-1}{2}}\right|^{\alpha} \,dxdt &= \fint_{\Omega_T} \left( \frac{\left|\nabla u^{\frac{m+q-1}{2}}\right|^2}{\left(A^{q-1}+u^{q-1}\right)^\xi} \right)^{\frac{\alpha}{2}} \left(A^{q-1}+u^{q-1}\right)^{\frac{\xi\alpha}{2}} \,dxdt\\
  &\le \left( \fint_{\Omega_T} \frac{\left|\nabla u^{\frac{m+q-1}{2}}\right|^2}{\left(A^{q-1}+u^{q-1}\right)^\xi} \,dxdt \right)^{\frac{\alpha}{2}} \underbrace{\left( \fint_{\Omega_T} \left(A^{q-1}+u^{q-1}\right)^{\frac{\xi\alpha}{2-\alpha}} \,dxdt \right)^{\frac{2-\alpha}{2}}}_{=:I}.
\end{aligned}
\end{equation*}
Then we have
\begin{equation*}
  I \le c(\alpha) A^{\frac{\xi\alpha(q-1)}{2}} + \left( \fint_{\Omega_T} u^{\frac{\xi\alpha(q-1)}{2-\alpha}} \,dxdt \right)^{\frac{2-\alpha}{2}} = c(\alpha) A^{\frac{\xi\alpha(q-1)}{2}},
\end{equation*}
by choosing the constant $A$ with
\begin{equation}\label{est03-01}
  A^{\frac{\xi\alpha(q-1)}{2}} = \left( \fint_{\Omega_T} u^{\frac{\xi\alpha(q-1)}{2-\alpha}} \,dxdt \right)^{\frac{2-\alpha}{2}}.
\end{equation}
We combine the above inequalities and \eqref{Estimate02} to discover
\begin{equation*}\label{est03-02}
\fint_{\Omega_T} \left|\nabla u^{\frac{m+q-1}{2}}\right|^{\alpha} \,dxdt
\le \left[\frac{(m+q-1)^2}{m(q-1)(\xi-1)}\right]^{\frac{\alpha}{2}} A^{\frac{\alpha(q-1)}{2}} \left[ \frac{2\mu(\Omega_{T})}{|\Omega_T|} + \frac{(q-1)(\xi-1)}{3m}\fint_{\Omega_T} u^{2-m} |V|^2 \,dxdt\right]^{\frac{\alpha}{2}}.
\end{equation*}
On the other hand, Lemma \ref{T:pSobolev} and \eqref{Estimate01} imply
\begin{equation}\label{est03-03}
\begin{aligned}
  \fint_{\Omega_T} \left(u^{\frac{m+q-1}{2}}\right)^{\alpha\frac{d+\frac{2}{m+q-1}}{d}} \,dxdt
  &\le c \fint_{\Omega_T} \left|\nabla u^{\frac{m+q-1}{2}}\right|^{\alpha} \,dxdt \left(\sup_{t \in (0, T)} \int_{\Omega} \left(u^{\frac{m+q-1}{2}}\right)^{\frac{2}{m+q-1}} \,dx \right)^{\frac{\alpha}{d}}\\
  &\le c [\mu(\Omega_{T})]^{\frac{\alpha}{d}} \fint_{\Omega_T}  \left|\nabla u^{\frac{m+q-1}{2}}\right|^{\alpha} \,dxdt,
\end{aligned}
\end{equation}
where $c=c(m,d,q,\alpha)\ge 1$.
Now we choose $\xi >1$ such that
\begin{equation*}
  \frac{\xi\alpha(q-1)}{2-\alpha} = \frac{m+q-1}{2}\alpha\frac{d+\frac{2}{m+q-1}}{d},
\end{equation*}
that is,
\begin{gather*}
  \xi = \frac{(2-\alpha)\{2+d(m+q-1)\}}{2d(q-1)} > 1 \iff \alpha < 2 - \frac{2d(q-1)}{2+d(m+q-1)} = \frac{2(2+md)}{2+d(m+q-1)},\\
  \xi-1=\frac{2(2+md)-\alpha\{2+d(m+q-1)\}}{2d(q-1)},\\
  \frac{(m+q-1)^2}{m(q-1)(\xi-1)} = \frac{2d(m+q-1)^2}{m[2(2+md)-\alpha\{2+d(m+q-1)\}]}.
\end{gather*}
Inserting \eqref{est03-03} into \eqref{est03-01}, we find
\begin{equation*}
   A^{\frac{\alpha(q-1)}{2}} \le c [\mu(\Omega_{T})]^{\frac{\alpha(q-1)}{2+d(m+q-1)}} \left(
\fint_{\Omega_T}
\left|\nabla u^{\frac{m+q-1}{2}}\right|^\alpha\,dxdt
\right)^{\frac{d(q-1)}{2+d(m+q-1)}}.
\end{equation*}
Combining the above estimates, we obtain
\begin{equation*}\label{est03-04}
\begin{aligned}
&\left(\fint_{\Omega_T} \left|\nabla u^{\frac{m+q-1}{2}}\right|^{\alpha} \,dxdt\right)^{1-\frac{d(q-1)}{2+d(m+q-1)}} \\
&\qquad \le c \left[\frac{2d(m+q-1)^2}{m[2(2+md)-\alpha\{2+d(m+q-1)\}]}\right]^{\frac{\alpha}{2}} [\mu(\Omega_{T})]^{\frac{\alpha(q-1)}{2+d(m+q-1)}} \left[ \frac{2\mu(\Omega_{T})}{|\Omega_T|}\right]^{\frac{\alpha}{2}}\\
&\qquad \quad + c [\mu(\Omega_{T})]^{\frac{\alpha(q-1)}{2+d(m+q-1)}} \left[\frac{(m+q-1)^2}{m^2}\fint_{\Omega_T} u^{2-m} |V|^2 \,dxdt\right]^{\frac{\alpha}{2}},
\end{aligned}
\end{equation*}
where $c=c(m,d,q,\alpha)\ge 1$. This implies the estimate \eqref{Estimate03}.
\end{proof}

We remark that the previous results can also be obtained under nonzero initial data.
\begin{remark}\label{L2 rmk}
Suppose that the initial datum is a finite nonnegative measure $\mu_0$ on $\Omega$. Then the estimates in Lemma~\ref{comparison estimates} remain valid with $\mu(\Omega_T)$ replaced by
$\nu(\Omega_T):=\mu(\Omega_T)+\mu_0(\Omega)$. See also Remark~\ref{L1 rmk}.
\end{remark}

We now finalize the gradient estimate under the appropriate assumptions on $V$.

\begin{theorem}\label{T:final_est}
Assume the same hypotheses as in Lemma~\ref{L1} for parts (i) and (ii), and as in Lemma~\ref{L:divfree:apriori} for part (iii).

\begin{enumerate}[(i)]
	\item (Porous medium case) If $1\leq m < 2$ and $V$ satisfies \eqref{T:V:PME}, then for any $\alpha \in (1, 2)$, the following estimate holds:
	 \begin{equation}\label{E_V}
	\left(\fint_{\Omega_T} \left|\nabla u^{\frac{m}{2}}\right|^{\alpha} \,dxdt\right)^{\frac{1}{\alpha}}
\leq c\left[ \frac{\mu(\Omega_{T})}{|\Omega_T|}\right]^{\frac{1}{2}}
 + C \|V\|_{L_{x,t}^{q_1, q_2}}^{\sigma_1} \left[\mu(\Omega_{T})\right]^{\sigma_2},
   \end{equation}
   where $c=c(m,d,\alpha)$, $C=C(m,d,\alpha,|\Omega_{T}|)$, $\sigma_1 = \frac{2q_2}{(2-\alpha)q_2 +2 \alpha}$, and $\sigma_2=\frac{2(2-m)q_2-\alpha m (q_2-2)}{2[(2-\alpha)q_2 +2\alpha]}$.

	Furthermore, if $m=2$ and $V$ satisfies \eqref{T:V:PME_2}, then for any $\alpha \in (1,2)$, we have 
	\begin{equation}\label{E_V_2}
	\left(\fint_{\Omega_T} \left|\nabla u^{\frac{m}{2}}\right|^{\alpha} \,dxdt\right)^{\frac{1}{\alpha}}
\leq c\left[ \frac{\mu(\Omega_{T})}{|\Omega_T|}\right]^{\frac{1}{2}}
 + C \|V\|_{L_{x,t}^{q_1, q_2}},
   \end{equation}
   where $c=c(m,d,\alpha)$, $C=C(m,d,\alpha,|\Omega_{T}|)$. 
   
 \smallskip

 \item (Fast diffusion case) If $1-\frac{1}{d} < m < 1$ and $V$ satisfies \eqref{T:V:FDE}, then for any $\alpha \in (1, 2)$, the estimate \eqref{E_V} holds.

 \smallskip

 \item (Divergence-free case) If $m > 0$ and $\nabla \cdot V = 0$, then for any $\alpha \in (1,2)$, the following estimate holds: 
     \begin{equation}\label{E_V_divfree}
	  \left(\fint_{\Omega_T} \left|\nabla u^{\frac{m}{2}}\right|^{\alpha} \,dxdt\right)^{\frac{1}{\alpha}}
  \leq c\left[ \frac{\mu(\Omega_{T})}{|\Omega_T|}\right]^{\frac{1}{2}},
   \end{equation}
   where $c=c(m,d,\alpha)$.
   \end{enumerate}
\end{theorem}

\begin{proof} 
Now we handle the term of $V$ in the estimate \eqref{Estimate04}. Let us set
    \[
   \mathcal{V}= \int_{\Omega_T} u^{2-m} |V|^2 \,dxdt.
   \]

If $m=2$, then $\mathcal{V}= \|V\|^{2}_{L_{x,t}^{2,2}}$. Therefore, the estimate \eqref{Estimate04} directly becomes 
\begin{equation}\label{E_V_22}
\left(\fint_{\Omega_T} \left|\nabla u^{\frac{m}{2}}\right|^{\alpha} \,dxdt\right)^{\frac{1}{\alpha}}
\le c \left[ \frac{\mu(\Omega_{T})}{|\Omega_T|}\right]^{\frac{1}{2}}+ c \left(\fint_{\Omega_T} |V|^2 \,dxdt\right)^{\frac{1}{2}}.
\end{equation}

Now let $1-\frac{1}{d}< m <2$. Applying H\"{o}lder inequalities in space with $q_1 >2$ and time with $q_2 >2$, we obtain 
\begin{equation*}%\label{V:est:01}
	\mathcal{V} \leq \|V\|_{L_{x,t}^{q_1, q_2}}^{2} \|u\|_{L_{x,t}^{r_1, r_2}}^{2-m},
	\quad  r_1 = \frac{(2-m)q_1}{q_1 - 2}, \quad 	r_2 = \frac{(2-m)q_2}{q_2 - 2}.
	\end{equation*}
Requiring the pair $(r_1, r_2)$ to satisfy \eqref{Lr1r2} in Lemma~\ref{L:V} is equivalent to  
\begin{equation}\label{V_alpha}
\begin{gathered}
V \in L_{x,t}^{q_1,q_2} \ \text{ where } \frac{d}{q_1} + \frac{\alpha (2+md) - 2d}{2q_2} = \frac{\alpha(2+md)+2d(m-2)}{4},	 \\
 \text{ for }
 \begin{cases}
 	\frac{m-1}{2} \leq \frac{1}{q_1} \leq \frac{1}{2}-\frac{(2-m)(d-\alpha)}{\alpha md}, \
 	\frac{1}{2}-\frac{2-m}{\alpha m} \leq \frac{1}{q_2} \leq \frac{1}{2}  &\text{if } 1\leq m < 2, \ \frac{2d}{2+md} \leq \alpha < 2, \vspace{1mm} \\
 	0\leq \frac{1}{q_1} \leq \frac{\alpha(2+md)+2d(m-2)}{4d}, \
 	0 \leq \frac{1}{q_2} \leq \frac{1}{2}-\frac{d(1-m)}{\alpha(2+md)-2d} & \text{if } 1-\frac{1}{d}< m < 1, \ \frac{2d(2-m)}{2+md} \leq \alpha < 2.
 \end{cases}
\end{gathered}\end{equation}
Then, combining \eqref{Lr1r2_norm} in Lemma~\ref{L:V} and \eqref{Estimate01} in Lemma~\ref{L1} yields 
\begin{equation}\label{V:est:02}
\mathcal{V} \leq c\|V\|_{L_{x,t}^{q_1, q_2}}^{2} \left[\mu(\Omega_{T})\right]^{\frac{2(2-m) q_2 - \alpha m (q_2 -2)}{2q_2}} \left(\int_{\Omega_T} \left|\nabla u^{\frac{m}{2}}\right|^{\alpha}\,dxdt\right)^{1-\frac{2}{q_2}}. 
\end{equation}

We now derive \eqref{E_V} for each case.

$\bullet$ (\emph{Porous medium case}) 
If $m=2$, the desired estimate \eqref{E_V_2} follows immediately from \eqref{E_V_22} by applying H\"{o}lder inequalities for any $q_1, q_2 \geq 2$.

Let $1\leq m <2$ and $V \in L_{x,t}^{\frac{2}{m-1}, 2}$. In this setting, $V$ belongs to the scaling-invariant class $\mathcal{S}_{m}^{(q_1, q_2)}$ (with the convention that $\frac{2}{m-1}=\infty$ if $m=1$). Applying H\"{o}lder inequality in space with $(2-m)+(m-1)=1$ and using \eqref{Estimate01}, we have
\begin{equation}\label{V:est:01}
\mathcal{V} \leq \left(\sup_{t\in (0,T)} \int_{\Omega} u \,dx\right)^{2-m} \|V\|_{L_{x,t}^{\frac{2}{m-1}, 2}}^{2} \leq [\mu(\Omega_T)]^{2-m} \cdot \|V\|_{L_{x,t}^{\frac{2}{m-1}, 2}}^{2}.
\end{equation}
Note that \eqref{V:est:01} is also valid for $m=1$. This gives \eqref{E_V} with $\sigma_1 = 1$ and $\sigma_{2}=\frac{2-m}{2}$ since $q_2 =2$.

If $1\leq m <2$ and $V$ satisfies \eqref{V_alpha}, inserting \eqref{V:est:02} into \eqref{Estimate04} implies 
 \begin{equation}\label{V:est:03}
\left(\fint_{\Omega_T} \left|\nabla u^{\frac{m}{2}}\right|^{\alpha} \,dxdt\right)^{\frac{1}{\alpha}}
\le c \left[ \frac{\mu(\Omega_{T})}{|\Omega_T|}\right]^{\frac{1}{2}}
+ c \|V\|_{L_{x,t}^{q_1, q_2}} \frac{\left[\mu(\Omega_{T})\right]^{\frac{2(2-m)q_2 - \alpha m (q_2 -2)}{4q_2}}}{|\Omega_{T}|^{\frac{1}{q_2}}} 
\left(\fint_{\Omega_T} \left|\nabla u^{\frac{m}{2}}\right|^{\alpha}\,dxdt\right)^{\frac{q_2 -2}{2q_2}}. 
	\end{equation}   
By applying Young's inequality with $\frac{\alpha(q_2 -2)}{2q_2} + \frac{(2-\alpha)q_2 + 2\alpha}{2q_2}=1$, we obtain 
\begin{equation*}\label{V:est:04}\begin{aligned}
\left(\fint_{\Omega_T} \left|\nabla u^{\frac{m}{2}}\right|^{\alpha} \,dxdt\right)^{\frac{1}{\alpha}}
&\le c \left[ \frac{\mu(\Omega_{T})}{|\Omega_T|}\right]^{\frac{1}{2}}
+ \frac{1}{2} \left(\fint_{\Omega_T} \left|\nabla u^{\frac{m}{2}}\right|^{\alpha} \,dxdt\right)^{\frac{1}{\alpha}} + c\|V\|_{L_{x,t}^{q_1, q_2}}^{\frac{2q_2}{(2-\alpha)q_2 + 2\alpha}} 
\frac{\left[\mu(\Omega_{T})\right]^{\frac{2(2-m)q_2-\alpha m(q_2-2)}{2\left[(2-\alpha)q_2 + 2\alpha\right]}}}{|\Omega_{T}|^{\frac{2}{(2-\alpha)q_2 + 2\alpha}}},
\end{aligned}\end{equation*}
which implies \eqref{E_V}. Since \eqref{Estimate04} is obtained for any $\alpha \in (0,2)$, the condition \eqref{V_alpha} is valid for $1\leq m <2$ and $\frac{2d}{2+md} \leq \alpha < 2$. As $\alpha \to 2$, \eqref{V_alpha} corresponds to 
\[
V \in \mathfrak{S}_{m}^{(q_1,q_2)} \ \text{ where } \ \frac{md}{(2-m)+d(m-1)} < q_1 \leq \frac{2}{m-1}, \  2 \leq q_2 < \frac{m}{m-1}.
\]

If $V$ satisfies \eqref{T:V:PME}, then $(q_1,q_2)$ either equals $\left( \frac{2}{m - 1}, 2 \right)$ (where \eqref{V:est:01} or \eqref{E_V_22} applies), or satisfies \eqref{V_alpha} (where \eqref{V:est:03} applies), or lies in the subcritical regime of \eqref{V_alpha}. In the subcritical case, there exists a pair $(\tilde{q}_1, \tilde{q}_2)$ satisfying \eqref{V_alpha} such that $\tilde{q}_1 \leq q_1$ and $\tilde{q}_2 \leq q_2$. Therefore, we have 
\begin{equation}\label{subSalpha}
\|V\|_{L^{\tilde{q}_1, \tilde{q}_2}_{x,t}} \leq |\Omega|^{\frac{1}{\tilde{q}_1} - \frac{1}{q_1}} T^{\frac{1}{\tilde{q}_2} - \frac{1}{q_2}} \|V\|_{L^{q_1, q_2}_{x,t}},
\end{equation}
which again leads to \eqref{V:est:03}. Thus \eqref{E_V} is valid for all $V$ satisfying \eqref{T:V:PME}.

$\bullet$ (\emph{Fast diffusion case}) Assume that $1-\frac{1}{d}<m<1$. By the same reasoning as above, the class \eqref{V_alpha} as $\alpha \to 2$ corresponds to \eqref{T:V:FDE}. Consequently, for any $V$ satisfying \eqref{T:V:FDE}, the estimate \eqref{E_V} holds.

$\bullet$ (\emph{Divergence-free case})  
Since both \eqref{Estimate01} and \eqref{Estimate02:divfree} are independent of $V$, the estimate \eqref{Estimate04} from Lemma~\ref{comparison estimates} contains no drift integral term. Therefore, it directly yields \eqref{E_V_divfree}.
\end{proof}

%%%%%%%%%%%%%%%%%%%%%%%%%%%%%%%%%%%%%%%%%%%%%%%
\section{Existence of weak solutions}\label{S:Exist}

In this section, we construct weak solutions to \eqref{PME} by means of compactness arguments applied to a sequence of regularized problems. As a preliminary step, we establish an estimate for the time derivative of a regular solution under suitable assumptions on the drift field $V$.

\begin{lemma}\label{time derivative}
Let $u$ be a regular solution of \eqref{PME} with $\mu\in L^1(\Omega_T)$. Assume that
\begin{equation}\label{u_FC}
\sup_{t\in (0, T)} \int_{\Omega} u \,dx < \infty \ \text{ and } \  \int_{\Omega_T} |\nabla u^{\frac{m}{2}}|^{\alpha} \,dxdt < \infty,
\end{equation}
where $\alpha$ satisfies 
\begin{equation}\label{alpha01}
\begin{cases}
	\alpha \in [\frac{2}{2-m},2), & \text{if } \ (1-\frac{2}{d})_{+} < m < 1, \vspace{1mm}\\
	\alpha \in [\frac{2(1+md)}{2+md},2), & \text{if }\ m \geq 1.
\end{cases}
\end{equation}
Furthermore, assume that $V$ satisfies either
\begin{equation}\label{V:uV01}
V \in L_{x,t}^{q_1,q_2} \ \text{ for } \ q_1 = \infty, \ q_2 =1, \quad \text{ if } \ m \geq 1,
\end{equation}
or
\begin{equation}\label{V:uV02}
\begin{aligned}
V \in L_{x,t}^{q_1,q_2}  & \text{ where } \ \frac{d}{q_1} + \frac{2+d(m-1)}{q_2} < 2+d(m-1), \vspace{1mm} \\
& \  \text{ for }\
\begin{cases}
 0 \leq \frac{1}{q_1} \leq \frac{2 + d(m-1)}{d}, \
0\leq \frac{1}{q_2} \leq 1  & \text{ if } \ (1-\frac{2}{d})_{+}< m < 1, \vspace{1mm}\\
 0 \leq \frac{1}{q_1} < \frac{2 + d(m-1)}{md}, \
\frac{m - 1}{m} < \frac{1}{q_2} \leq 1 & \text{ if } \ m \geq 1.
\end{cases}
\end{aligned}
\end{equation}
Then, it follows that
\[
\partial_t u \in L^{1}(0,T;  W^{-1,1}(\Omega)).
\]
\end{lemma}

\begin{proof}
We proceed under the condition \eqref{u_FC}, which is established in Lemma~\ref{L1} and Theorem~\ref{T:final_est}. 

To show that $\partial_t u \in L^{1}(0,T;  W^{-1,1}(\Omega))$, it suffices to bound the right-hand side of the weak formulation
\begin{equation}\label{compact01}
\int_{\Omega_T} \partial_t u \,\varphi \,dxdt
=
-\int_{\Omega_T}\nabla u^m\cdot\nabla\varphi\,dxdt
+
\int_{\Omega_T}uV\cdot\nabla\varphi\,dxdt
+
\int_{\Omega_T}\mu\varphi\,dxdt
\end{equation}
for every $\varphi \in C^{\infty}(\overline{\Omega} \times [0,T])$ such that $\varphi = 0$ on $\partial \Omega \times (0,T)$ and $\varphi(\cdot, T)=0$.

First, it is straightforward to see that 
\[
\int_{\Omega_T} \mu \varphi \,dxdt \leq \|\mu\|_{L^1(\Omega_T)}\|\varphi\|_{L^{\infty}(\Omega_T)} < \infty. 
\]

Next, using the identity $|\nabla u^m| = 2 u^{\frac{m}{2}}|\nabla u^{\frac{m}{2}}|$, we obtain 
\begin{equation*}
\int_{\Omega_T} |\nabla u^m| |\nabla\varphi| \,dxdt
\leq 2\|\nabla \varphi\|_{L^{\infty}(\Omega_{T})} \int_{\Omega_T} u^{\frac{m}{2}}|\nabla u^{\frac{m}{2}}| \,dxdt.
\end{equation*}
By H\"{o}lder's inequality, it follows that 
\begin{equation*}
\begin{aligned}
\int_{\Omega_T} u^{\frac{m}{2}}|\nabla u^{\frac{m}{2}}| \,dxdt 
&\leq \left(\int_{\Omega_T} |\nabla u^{\frac{m}{2}}|^{\alpha} \,dxdt\right)^{\frac{1}{\alpha}} \left(\int_{\Omega_T} u^{\frac{\alpha m}{2(\alpha -1)}} \,dxdt\right)^{1-\frac{1}{\alpha}}.
\end{aligned}
\end{equation*}
Consequently, the second integral is finite provided that the exponent satisfies 
\[
\frac{\alpha m}{2(\alpha -1)} \leq
\begin{cases}
	1, & \text{ for } (1-\frac{2}{d})_{+}<m < 1 \text{ and  } \alpha \in [\frac{2}{2-m},2), \vspace{1mm}\\
	\frac{\alpha (2+md)}{2d}, & \text{ for } m \geq 1 \text{ and } \alpha \in [\frac{2(1+md)}{2+md}, 2).
\end{cases}
\]
Indeed, by Lemma~\ref{T:pSobolev}, the assumption \eqref{u_FC} ensures that $u \in L_{x,t}^{\frac{\alpha(2+md)}{2d}}$, and in the singular case, we directly use the first condition in \eqref{u_FC}.

Lastly, applying H\"{o}lder's inequality with $q_1, q_2 > 1$ yields
\begin{equation}\label{compact02}
	\int_{\Omega_T} u |V| |\nabla \varphi| \,dxdt
	\leq \|\nabla \varphi\|_{L^{\infty}} \|V\|_{L_{x,t}^{q_1,q_2}} \|u\|_{L_{x,t}^{r_1,r_2}}, 
\end{equation}
where $r_1 = \frac{q_1}{q_1 -1}$ and $r_2 = \frac{q_2}{q_2 -1}$.
Note that the bound for the case $V\in L_{x,t}^{\infty,1}$ in \eqref{V:uV01} is immediate since $u\in L_{x,t}^{1,\infty}(\Omega_{T})$.
Furthermore, Lemma~\ref{L:V} combined with \eqref{u_FC} guarantees that $\|u\|_{L_{x,t}^{r_1,r_2}} < \infty$ whenever the pair $(r_1, r_2)$ satisfies \eqref{Lr1r2}, which corresponds to the pair $(q_1, q_2)$ in \eqref{V:uV02}. 

Combining the above estimates in \eqref{compact01}, we conclude that 
\[
\partial_t u \in L^{1}(0,T;  W^{-1,1}(\Omega)),
\]
which completes the proof.
\end{proof}

Now, we are ready to prove our main existence theorems (Theorems~\ref{T:PME} and \ref{T:divfree}).

\begin{proof}[Proof of Theorem~\ref{T:PME}]
Let $\mu$ be a nonnegative finite measure on $\Omega_T$. Then there exists a sequence of nonnegative functions $\{\mu_n\}\subset L^\infty(\Omega_T)$ such that
\[
\mu_n(\Omega_T):=\int_{\Omega_T}\mu_n\,dxdt \le \mu(\Omega_T)
\]
and
\[
\mu_n \rightharpoonup \mu
\qquad \text{weakly in the sense of measures.}
\]
Likewise, we choose a sequence $\{V_n\}\subset C^\infty(\Omega_{T})\cap L^\infty(\Omega_T)$ such that
\[
V_n \to V
\quad \text{strongly in } L_{x,t}^{q_1,q_2}(\Omega_T),
\]
and
\[
\sup_n \|V_n\|_{L_{x,t}^{q_1,q_2}(\Omega_T)}<\infty.
\]
When one of the exponents $q_1,q_2$ is infinite, this approximation is understood in finite lower exponents which are still admissible for the estimates below, together with the corresponding uniform endpoint bound. In particular, the endpoint $(q_1,q_2)=(\infty,2)$ for $m=1$ is handled by approximating in $L_{x,t}^{q,2}$ with a fixed sufficiently large finite $q$, while retaining the uniform $L_{x,t}^{\infty,2}$ bound.

For each $n$, let $u_n$ be a regular solution in the sense of Definition~\ref{reg sol} to the regularized problem
\begin{equation}\label{reg-PME}
\partial_t u_n-\Delta u_n^m+\nabla\cdot(u_nV_n)=\mu_n
\end{equation}
with zero initial and boundary data. Moreover, for the case of measure-valued initial data, a similar argument applies, see Remark~\ref{L0 rmk}.

Choose $\bar\alpha\in(1,2)$ sufficiently close to $2$ so that
\[
1<\kappa<\frac{\bar\alpha(2+md)}{2(d+1)}
\]
for some $\kappa>1$. This is possible because $m>1-\frac1d$.
Applying the estimates from Section~\ref{S:apriori-est} and Theorem~\ref{T:final_est}, we obtain
\[
\sup_{t\in(0,T)}\int_\Omega u_n(x,t)\,dx
\le
\mu_n(\Omega_T)
\le
\mu(\Omega_T),
\]
and
\[
\sup_n
\int_{\Omega_T}\left|\nabla u_n^{\frac m2}\right|^{\bar\alpha}\,dxdt
\le C.
\]
Moreover, Lemma~\ref{T:pSobolev} yields
\[
\sup_n \|u_n\|_{L^r(\Omega_T)}\le C,
\qquad
r:=\frac{\bar\alpha(2+md)}{2d}.
\]

Next, H\"{o}lder's inequality gives
\[
\begin{aligned}
\int_{\Omega_T} |\nabla u_n|^\kappa\,dxdt
&=
c(m)\int_{\Omega_T}
\left|\nabla u_n^{\frac m2}\right|^\kappa
u_n^{\frac{\kappa(2-m)}{2}}
\,dxdt \\
&\le
c(m)
\left(\int_{\Omega_T}\left|\nabla u_n^{\frac m2}\right|^{\bar\alpha}\,dxdt\right)^{\frac{\kappa}{\bar\alpha}}
\left(\int_{\Omega_T}u_n^{\frac{\bar\alpha \kappa(2-m)}{2(\bar\alpha-\kappa)}}\,dxdt\right)^{1-\frac{\kappa}{\bar\alpha}}.
\end{aligned}
\]
Since
\[
\frac{\bar\alpha \kappa(2-m)}{2(\bar\alpha-\kappa)}
\le
\frac{\bar\alpha(2+md)}{2d}=r
\qquad \Longleftrightarrow \qquad
\kappa\le \frac{\bar\alpha(2+md)}{2(d+1)},
\]
the right-hand side is uniformly bounded in $n$. Hence
\[
\sup_n \|\nabla u_n\|_{L^\kappa(\Omega_T)}<\infty.
\]
Because $u_n=0$ on $\partial\Omega\times(0,T)$, Poincar\'e's inequality implies
\[
\sup_n \|u_n\|_{L^\kappa(0,T;W_0^{1,\kappa}(\Omega))}<\infty.
\]

By Lemma~\ref{time derivative},
\[
\sup_n \|\partial_t u_n\|_{L^1(0,T;W^{-1,1}(\Omega))}<\infty.
\]
We now apply Lemma~\ref{AL} with
\[
X_0=W_0^{1,\kappa}(\Omega), \qquad X=L^1(\Omega), \qquad X_1=W^{-1,1}(\Omega).
\]
Since
\[
W_0^{1,\kappa}(\Omega)\Subset L^1(\Omega)
\qquad \text{and} \qquad
L^1(\Omega)\hookrightarrow W^{-1,1}(\Omega),
\]
we conclude that, up to a subsequence,
\[
u_n \to u
\quad \text{strongly in } L^1(\Omega_T)
\]
and almost everywhere in $\Omega_T$.

On the other hand, after passing to a further subsequence,
\[
\nabla u_n^{\frac m2}\rightharpoonup w
\quad \text{weakly in } L^{\bar\alpha}(\Omega_T;\mathbb R^d),
\]
and
\[
u_n \rightharpoonup u
\quad \text{weakly in } L^r(\Omega_T).
\]
By interpolation between the strong convergence in $L^1(\Omega_T)$ and the uniform bound in $L^r(\Omega_T)$,
\[
u_n \to u
\quad \text{strongly in } L^\gamma(\Omega_T)
\quad \text{for every } 1\le \gamma<r.
\]
In particular,
\[
u_n^{\frac m2}\to u^{\frac m2}
\quad \text{strongly in } L^{\frac{\bar\alpha}{\bar\alpha-1}}(\Omega_T).
\]

To identify the weak limit, let $\eta\in C_c^\infty(\Omega_T)$. Then
\[
\int_{\Omega_T} w \eta\,dxdt
=
\lim_{n\to\infty}
\int_{\Omega_T}\nabla u_n^{\frac m2}\eta\,dxdt
=
-\lim_{n\to\infty}
\int_{\Omega_T}u_n^{\frac m2}\,\nabla\eta\,dxdt.
\]
Since $u_n^{\frac m2}\to u^{\frac m2}$ strongly in $L^{\frac{\bar\alpha}{\bar\alpha-1}}(\Omega_T)$, we obtain
\[
\int_{\Omega_T} w\eta\,dxdt
=
-\int_{\Omega_T}u^{\frac m2}\,\nabla\eta\,dxdt.
\]
Hence
\[
w=\nabla u^{\frac m2}
\qquad \text{in the sense of distributions.}
\]

It remains to pass to the limit in the weak formulation for \eqref{reg-PME}. The time derivative term and the measure term are standard. For the diffusion term, we write
\[
\nabla u_n^m = 2u_n^{\frac m2}\nabla u_n^{\frac m2},
\]
and use the strong convergence of $u_n^{m/2}$ together with the weak convergence of $\nabla u_n^{m/2}$ to deduce
\[
\nabla u_n^m \rightharpoonup \nabla u^m
\quad \text{weakly in } L^1(\Omega_T;\mathbb R^d).
\]
For the drift term, let
\[
r_1=\frac{q_1}{q_1-1},
\qquad
r_2=\frac{q_2}{q_2-1}.
\]
Since $(r_1,r_2)$ satisfies \eqref{Lr1r2}, Lemma~\ref{L:V}, together with the uniform bounds above, yields
\[
\sup_n \|u_n\|_{L_{x,t}^{r_1,r_2}(\Omega_T)}<\infty.
\]
Hence, up to a subsequence,
\[
u_n \rightharpoonup u
\quad \text{weakly in } L_{x,t}^{r_1,r_2}(\Omega_T).
\]
Since $V_n\to V$ strongly in $L_{x,t}^{q_1,q_2}(\Omega_T)$, it follows that
\[
u_nV_n \rightharpoonup uV
\quad \text{weakly in } L^1(\Omega_T;\mathbb R^d).
\]
Passing to the limit in the weak formulation, we conclude that $u$ is a weak solution of \eqref{PME} in the sense of Definition~\ref{D:WS}.
\end{proof}

\begin{proof}[Proof of Theorem~\ref{T:divfree}]
We note that a naive mollification of a divergence-free vector field on $\Omega$ does not in general remain divergence-free on the whole domain $\Omega$. Therefore, instead of working directly on $\Omega$, we use an interior exhaustion.

Let $\{\Omega_n\}_{n=1}^\infty$ be a sequence of bounded smooth domains such that
\[
\overline{\Omega_n}\subset \Omega_{n+1}\Subset \Omega,
\qquad
\bigcup_{n=1}^\infty \Omega_n=\Omega.
\]
Set $
\Omega_{n,T}:=\Omega_n\times(0,T).$
Let $\mu$ be a nonnegative finite measure on $\Omega_T$. Then there exists a sequence of nonnegative functions
$
\{\mu_n\}\subset C_c^\infty(\Omega_{n,T})
$
such that
\[
\mu_n(\Omega_{n,T})=\int_{\Omega_{n,T}}\mu_n\,dxdt \le \mu(\Omega_T)
\]
and $\mu_n \rightharpoonup \mu$ weakly in the sense of measures on $\Omega_T$.
Let $\widetilde V$ denote the extension of $V$ by zero outside $\Omega_T$.
Choose $\varepsilon_n>0$ so that $
\varepsilon_n < \frac12 \operatorname{dist}(\Omega_n,\partial\Omega),
$
and let $\rho_{\varepsilon_n}^x$ and $\eta_{\varepsilon_n}^t$ be standard mollifiers in space and time.
Define
\[
V_n:=\eta_{\varepsilon_n}^t *_t \left(\rho_{\varepsilon_n}^x *_x \widetilde V\right)
\qquad \text{in } \Omega_{n,T}.
\]
Then
$
V_n\in C^\infty(\Omega_{n,T})\cap L^\infty(\Omega_{n,T}),
$
and
$
V_n \to V
$ strongly in the ambient space corresponding to either \eqref{V:uV01} or \eqref{V:uV02}, locally in $\Omega_T$.
Moreover, \( \nabla\cdot V_n=0 \) in $\Omega_{n,T}$ for every $n$.

For each $n$, let $u_n$ be a regular solution in $\Omega_{n,T}$ of
\begin{equation}\label{reg-PME_n}
\partial_t u_n-\Delta u_n^m+\nabla\cdot(u_nV_n)=\mu_n
\qquad \text{in } \Omega_{n,T},
\end{equation}
with zero initial data and zero lateral boundary data on
\(
\partial\Omega_n\times(0,T).
\)
We extend $u_n$ by zero from $\Omega_{n,T}$ to $\Omega_T$, and still denote the extension by $u_n$.

Fix $\bar\alpha\in(1,2)$ such that
\begin{equation}\label{alpha-star-divfree}
\bar\alpha>
\max\left\{
\frac{2d}{2+md},
\,
\frac{2}{2-m} \chi_{\{m<1\}},
\,
\frac{2(1+md)}{2+md} \chi_{\{m\ge1\}}
\right\}.
\end{equation}
This is possible since $m>(1-\frac{2}{d})_+$.
By Theorem~\ref{T:final_est}\,\textup{(iii)}, applied on $\Omega_{n,T}$, we obtain
\begin{equation}\label{divfree-est-1}
\sup_{t\in(0,T)}\int_{\Omega_n} u_n(x,t)\,dx
\le
\mu_n(\Omega_{n,T})
\le
\mu(\Omega_T),
\end{equation}
and
\begin{equation}\label{divfree-est-2}
\sup_n
\int_{\Omega_{n,T}}
\left|\nabla u_n^{\frac m2}\right|^{\bar\alpha}\,dxdt
\le C.
\end{equation}
Since $\bar\alpha>\frac{2d}{2+md}$, Lemma~\ref{T:pSobolev} yields
\begin{equation}\label{divfree-est-3}
\sup_n \|u_n\|_{L^p(\Omega_{n,T})}\le C,
\qquad
p:=\frac{\bar\alpha(2+md)}{2d}>1.
\end{equation}

We now prove compactness. The argument is divided according to the diffusion
range.

\smallskip
\noindent
\underline{\emph{Case 1: $1\le m\le2$.}}
Choose $\kappa$ such that
\[
1<\kappa<\frac{\bar\alpha(2+md)}{2(d+1)}.
\]
Exactly as in the proof of Theorem~\ref{T:PME}, H\"{o}lder's inequality gives
\[
\sup_n \|\nabla u_n\|_{L^\kappa_{\rm loc}(\Omega_T)}<\infty.
\]
Hence, for every $U\Subset\Omega$ and every $0<t_1<t_2<T$,
\[
\sup_n \|u_n\|_{L^\kappa(t_1,t_2;W^{1,\kappa}(U))}<\infty.
\]
Moreover, by Lemma~\ref{time derivative},
\[
\sup_n \|\partial_t u_n\|_{L^1(t_1,t_2;W^{-1,1}(U))}<\infty.
\]
Therefore, by Lemma~\ref{AL}, after passing to a subsequence,
\[
u_n\to u
\quad \text{strongly in }L^1_{\rm loc}(\Omega_T)
\]
and almost everywhere in $\Omega_T$.

\smallskip
\noindent
\underline{\emph{Case 2: $\left(1-\frac{2}{d}\right)_+<m<1$.}}
For $k\in\mathbb N$, define
\[
w_n^{(k)}:=T_k(u_n):=\min\{u_n,k\}.
\]
Since $0<m<1$, we have
\[
|\nabla w_n^{(k)}|
=
\frac{2}{m}\left(w_n^{(k)}\right)^{1-\frac m2}
\left|\nabla \left(w_n^{(k)}\right)^{\frac m2}\right|
\le
C(k,m)\left|\nabla u_n^{\frac m2}\right|.
\]
Hence, for every $U\Subset\Omega$ and every $0<t_1<t_2<T$,
\begin{equation}\label{divfree-trunc-grad}
\sup_n \|w_n^{(k)}\|_{L^{\bar\alpha}(t_1,t_2;W^{1,\bar\alpha}(U))}\le C(k,U,t_1,t_2).
\end{equation}
At this point we use the truncation theory from Sections~3 and 4.3 in \cite{BDG15}. In the present divergence-free setting, the drift term disappears in the truncation estimates exactly as in the proof of Lemma~\ref{L:divfree:apriori}. Therefore the arguments of Lemmas~3.1--3.4 and Section~4.3.1--4.3.2 in \cite{BDG15} apply to the truncations $T_k(u_n)$ and yield the local time-derivative estimate
\begin{equation}\label{divfree-trunc-time}
\sup_n
\|\partial_t w_n^{(k)}\|_{L^1(t_1,t_2;W^{-1,1}(U))}
\le
C(k,U,W,t_1,t_2, \|V\|_{L^{1}(\Omega_T)})
\end{equation}
for every $U\Subset W\Subset\Omega$ and every $0<t_1<t_2<T$. Consequently, for every fixed $k\in\mathbb N$, Lemma~\ref{AL} implies that, after passing to a subsequence,
\[
T_k(u_n)\to T_k(u)
\quad \text{strongly in }L^1_{\rm loc}(\Omega_T).
\]
Repeating the diagonal construction exactly as in Section~4.3.2 of \cite{BDG15}, we conclude that
\[
u_n\to u
\quad \text{strongly in }L^1_{\rm loc}(\Omega_T)
\]
and almost everywhere in $\Omega_T$.

\smallskip
\noindent
\underline{\emph{Case 3: $m>2$.}}
In this case, the direct estimate of $\nabla u_n$ used in Case~1 is not
available because $u_n^{1-\frac m2}$ is singular near the zero set of
$u_n$. We therefore appeal to the compactness theory for the porous medium
equation with measure data in \cite{BDG13}. As in Case~2, the
divergence-free condition allows us to rewrite the drift contribution in
the local truncation estimates as a lower-order cutoff term. This term is
controlled locally by the integrability of $V_n$, and therefore it does not
affect the compactness argument.

Consequently, the local compactness method of \cite{BDG13} applies to the
sequence $\{u_n\}$ and gives, up to a subsequence,
\[
u_n\to u
\quad \text{strongly in }L^1_{\rm loc}(\Omega_T)
\]
and almost everywhere in $\Omega_T$.

\smallskip

In all cases, we have now obtained
\[
u_n\to u
\quad \text{strongly in }L^1_{\rm loc}(\Omega_T)
\]
and almost everywhere in $\Omega_T$.
Interpolating with \eqref{divfree-est-3}, we also obtain
\begin{equation}\label{divfree-strong-gamma}
u_n\to u
\quad \text{strongly in }L^\gamma(\Omega_T)
\quad \text{for every }1\le \gamma<p.
\end{equation}

From \eqref{divfree-est-2}, after passing to a further subsequence if necessary, we may assume that
\[
\nabla u_n^{\frac m2}\rightharpoonup z
\quad \text{weakly in }L^{\bar\alpha}_{\rm loc}(\Omega_T;\mathbb R^d).
\]
Since \eqref{alpha-star-divfree} implies
\(
\frac{m\bar\alpha}{2(\bar\alpha-1)}<p,
\)
the strong convergence \eqref{divfree-strong-gamma} yields
\[
u_n^{\frac m2}\to u^{\frac m2}
\quad \text{strongly in }L^{\frac{\bar\alpha}{\bar\alpha-1}}_{\rm loc}(\Omega_T).
\]
Let $\eta\in C_c^\infty(\Omega_T)$. Then
\[
\int_{\Omega_T} z\eta\,dxdt
=
\lim_{n\to\infty}\int_{\Omega_T}\nabla u_n^{\frac m2}\eta\,dxdt
=
-\lim_{n\to\infty}\int_{\Omega_T}u_n^{\frac m2}\,\nabla\eta\,dxdt
=
-\int_{\Omega_T}u^{\frac m2}\,\nabla\eta\,dxdt.
\]
Hence $z=\nabla u^{\frac m2}$ in the sense of distributions.

Let $\varphi \in C^{\infty} \left(\overline{\Omega} \times [0, T]\right)$ which vanishes on $\partial \Omega \times (0,T)$ and $\varphi(\cdot,T)=0$. 
We cannot use $\varphi$ directly as a test function in the approximate
problem on $\Omega_n$, since $\varphi$ does not necessarily vanish on
$\partial\Omega_n$. By using a suitable spatial cutoff function, we can construct a sequence of $\{\phi_j\} \subset C_c^\infty(\Omega\times[0,T))$ with $\phi_j(\cdot,T)=0$ such that $\varphi_j\to\varphi$ in $C^1_{\rm loc}(\Omega_T)$, and
$\varphi_j$, $(\varphi_j)_t$, $\nabla\varphi_j$ are uniformly bounded. Then for each fixed $j$, the function
$\varphi_j$ is admissible in the weak formulation for $u_n$ on
$\Omega_{n,T}$ for all sufficiently large $n$.
Testing the approximate equation by $\varphi_j$ and passing to the limit
$n\to\infty$, exactly as in the proof of Theorem~\ref{T:PME}, gives
\[
\int_{\Omega_T}
\left\{
-u(\varphi_j)_t+\nabla u^m\cdot\nabla\varphi_j
-uV\cdot\nabla\varphi_j
\right\}\,dxdt
=
\int_{\Omega_T}\varphi_j\,d\mu .
\]
Finally, letting $j\to\infty$, we obtain that $u$ is a weak solution of \eqref{PME} in $\Omega_T$. This completes the proof.
\end{proof}

%%%%%%%%%%%%%%%%%%%%%%%%%%%%%%%%%%

\section{Sharpness of the general drift threshold}
\label{S:sharpness-divfree}

In this section, we discuss two complementary phenomena concerning general
non-divergence-free drift fields. First, we show that the lower threshold
\[
m>1-\frac{1}{d}
\]
in the general drift case is sharp. More precisely, in dimension \(d=2\),
the improved range available in the divergence-free case cannot be extended
to arbitrary drift fields by imposing only the corresponding integrability
condition on \(V\).

Second, we record an additional obstruction in the highly degenerate range
\(m>2\). This obstruction is different from the sharpness of the lower
threshold. It shows that, without the divergence-free cancellation, even a
rather integrable drift may cancel the diffusion term and destroy the
gradient estimate for \(u^{\frac{m}{2}}\). This should be interpreted as a
counterexample to the gradient estimate, not as a nonexistence result.

\subsection{Sharpness of the lower threshold \(m>1-\frac{1}{d}\)}
\label{S:sharpness-lower}

We first recall that, even in the drift-free case \(V=0\), the lower bound
\[
m>\left(1-\frac{2}{d}\right)_+
\]
is intrinsic to the measure data theory. This is reflected by the Barenblatt solutions for the porous medium equation with a Dirac mass on the right-hand side; see \cite{BDG13,BDG15}. Thus the lower bound in Theorem~\ref{T:divfree} cannot be improved in general, even without drift.

The main point of this subsection is different. We show that, once a general non-divergence-free drift is allowed, the stronger lower bound
\[
m>1-\frac{1}{d}
\]
in Theorem~\ref{T:PME} is sharp. More precisely, the improved range in Theorem~\ref{T:divfree} cannot be extended to arbitrary drift fields by imposing only the same integrability scale on \(V\).

For simplicity, we present the construction only in dimension \(d=2\). In this case, Theorem~\ref{T:divfree} allows the whole range
\(
m>\left(1-\frac{2}{d}\right)_+=0
\)
when \(V\) is divergence-free, whereas Theorem~\ref{T:PME} for a general drift field requires
\(
m>1-\frac{1}{d}=\frac{1}{2}.
\)
The following example shows that this gap is not merely a technical artifact of the proof. For every
\[
0<m\le \frac{1}{2},
\]
we construct a non-divergence-free drift field satisfying the same isotropic integrability scale as in the divergence-free theory, together with a nonnegative finite Radon measure on the right-hand side, for which the gradient estimate
\[
\nabla u^{\frac{m}{2}}\in L^\alpha
\]
fails for some \(\alpha\in(0,2)\).

Let us first explain the choice of the integrability exponent for the drift.
In
dimension \(d=2\), the divergence-free admissible condition
\eqref{sub_S_divfree}, in the isotropic case 
\((q_1,q_2)=(p,p)\), becomes
\begin{equation}\label{isotropic_p}
\frac{2}{p}+\frac{2m}{p}<2m
\quad \Longleftrightarrow \quad
p>\frac{m+1}{m}.
\end{equation}
Thus, if a drift field \(V\in L^p_{x,t}\) with
\(p>\frac{m+1}{m}\) were divergence-free, then it would satisfy the
corresponding isotropic admissible condition in Theorem~\ref{T:divfree}.
In the example below, we choose \(p\) exactly in this range and construct a
drift field
\[
V\in L^p(Q_1)
\]
which is not divergence-free. Therefore, the failure of the gradient estimate
is not caused by insufficient integrability of \(V\), but by the absence of
the cancellation structure \(\nabla\cdot V=0\).

\begin{example}
\label{ex:non-divfree-counterexample}
Let
\[
D_1:=\{x\in\mathbb R^2: |x|<1\},
\qquad
Q_1:=D_1\times(0,1).
\]
Assume that
\[
0<m\le \frac{1}{2}.
\]
Let \(p\) be any finite exponent in the admissible range
\[
p>\frac{m+1}{m},
\]
and set
\[
q:=2-\frac{2}{p},
\qquad
s:=qm.
\]
Then \(1<q<2\). In particular, \(r^{-q}\) is integrable near the origin in
two dimensions.

We shall construct a nonnegative function \(u\), a non-divergence-free drift
field \(V\in L^p(Q_1)\), and a nonnegative finite Radon measure
\(\mu\) such that
\[
u_t-\Delta u^m+\nabla\cdot(uV)=\mu
\qquad
\text{in } \, Q_1,
\]
with zero initial and lateral boundary data, but
\[
\int_{Q_1}
\left|\nabla u^{\frac{m}{2}}\right|^{\alpha_0}\,dxdt
=
\infty
\]
for some \(\alpha_0\in(0,2)\).

Choose \(L>s\), and define
\[
g(r):=r^{-s}-r^L,
\qquad
h(r):=g(r)^{\frac{1}{m}},
\qquad
0<r<1.
\]
Since \(r^{-s}>1\) and \(r^L<1\) for \(0<r<1\), we have \(g(r)>0\).
Moreover,
\(
g(1)=0,
\)
and near \(r=0\),
\[
g(r)\sim r^{-s},
\qquad
h(r)\sim r^{-q}.
\]
Near the boundary \(r=1\), since
\(
g'(1)=-(s+L)<0,
\)
we have
\[
g(r)\sim (s+L)(1-r),
\qquad
h(r)\sim C(1-r)^{\frac{1}{m}}.
\]
Thus \(h\in L^1(D_1)\), \(h>0\) in \(D_1\), and \(h=0\) on \(\partial D_1\).

Next choose
\[
0<\gamma<\frac{1}{p(1-m)}
\]
and set
\[
a(t):=e^t t^\gamma,
\qquad
u(x,t):=a(t)h(|x|).
\]
Then
\[
u(\cdot,0)=0
\qquad
\text{in }L^1(D_1),
\]
because \(a(t)\to0\) as \(t\searrow0\) and \(h\in L^1(D_1)\). Also,
\[
u=0
\qquad
\text{on }\partial D_1\times(0,1),
\]
because \(h(r)\to0\) as \(r\nearrow1\).

We now construct the drift field. Define
\[
B(r):=-\int_r^1 \tau h(\tau)\,d\tau,
\qquad
W_0(x):=\frac{B(|x|)}{|x|^2}x.
\]
Writing \(r=|x|\) and \(e_r=\frac{x}{|x|}\), this becomes
\[
W_0(x)=\frac{B(r)}{r}e_r.
\]
Since \(h\in L^1(D_1)\), the number
\[
B(0)=-\int_0^1 \tau h(\tau)\,d\tau
\]
is finite and negative.

For \(r>0\), the two-dimensional radial divergence formula gives
\[
\nabla\cdot(A(r)e_r)
=
\frac{1}{r}\frac{d}{dr}\left(rA(r)\right).
\]
Taking \(A(r)=\frac{B(r)}{r}\), we obtain
\[
\nabla\cdot W_0
=
\frac{1}{r} B'(r)
=
h(r)
\qquad
\text{in }D_1\setminus\{0\},
\]
because \(B'(r)=rh(r)\).

However, \(W_0\) has a singularity at the origin. In the sense of
distributions on \(D_1\), we have
\[
\nabla\cdot W_0=h-M\delta_0,
\]
where \(\delta_0\) denotes the Dirac mass at the origin in the spatial variable and 
\[
M:=2\pi\int_0^1 r h(r)\,dr>0.
\]
Indeed, for every \(\psi\in C_c^\infty(D_1)\), integrating by parts on
\(D_1\setminus\overline{B_\varepsilon}\), using \(B(1)=0\), and then
letting \(\varepsilon\searrow0\), gives
\[
-\int_{D_1}W_0\cdot\nabla\psi\,dx
=
\int_{D_1}h\psi\,dx
+
2\pi B(0)\psi(0).
\]
Since \(2\pi B(0)=-M\), the distributional identity follows.

We next choose a radial cutoff function. Since
\[
\Delta g
=
g''+\frac{1}{r} g'
=
s^2r^{-s-2}-L^2r^{L-2},
\]
and \(L>s\), we have
\[
\Delta g(1)=s^2-L^2<0.
\]
Hence there exists \(R\in(0,1)\) such that
\[
\Delta g\le0
\qquad
\text{on }[R,1).
\]
Choose \(\eta\in C^\infty([0,1))\) such that
\[
0\le\eta\le1,
\qquad
\eta=1\quad\text{on }[0,R],
\qquad
\eta=0\quad\text{near }1,
\qquad
\eta'\le0.
\]
Define
\[
V(x,t):=
\frac{\eta(|x|)\nabla u^m(x,t)-a(t)W_0(x)}{u(x,t)}.
\]
Then
\[
uV=\eta\nabla u^m-a(t)W_0.
\]

We claim first that
\[
V\in L^p(Q_1).
\]
Since
\(
u^m=a(t)^m g(r),
\)
we have
\[
\nabla u^m=a(t)^m\nabla g.
\]
Therefore
\[
V
=
\eta a(t)^{m-1}\frac{\nabla g}{h}
-
\frac{W_0}{h}.
\]

Near \(r=0\), we have
\[
h(r)\sim r^{-q},
\qquad
|\nabla g(r)|\sim r^{-s-1}=r^{-qm-1}.
\]
Hence
\[
\frac{|\nabla g|}{h}
\sim
r^{-1+q(1-m)}.
\]
Thus the spatial \(L^p\)-integral near \(r=0\) behaves like
\[
\int_0 r^{1-p+pq(1-m)}\,dr.
\]
This integral is finite if
\(
p\left(1-q(1-m)\right)<2.
\)
Since \(q=2-\frac{2}{p}\) and \(0<m\le\frac{1}{2}\), we have 
\[
p\left(1-q(1-m)\right)=p(-1+2m)+2(1-m)\le2(1-m)<2.
\]
Therefore the spatial singularity of \(\frac{\nabla g}{h}\) is \(L^p\)-integrable.

The corresponding time factor is
\[
a(t)^{m-1}
=
e^{(m-1)t}t^{-\gamma(1-m)}.
\]
Thus
\[
a^{m-1}\in L^p(0,1),
\]
because
\(
\gamma p(1-m)<1.
\)
Consequently,
\[
\eta a(t)^{m-1}\frac{\nabla g}{h}\in L^p(Q_1).
\]

It remains to estimate \(\frac{W_0}{h}\). Near \(r=0\), since \(B(r)\to B(0)\ne0\),
\[
|W_0(x)|\sim \frac{C}{r}.
\]
Therefore
\[
\frac{|W_0|}{h}
\sim
r^{q-1}.
\]
Since \(q>1\), this is bounded near the origin.

Near \(r=1\), we have
\[
h(r)\sim C(1-r)^{\frac{1}{m}}.
\]
Moreover,
\[
B(r)
=
-\int_r^1 \tau h(\tau)\,d\tau
\sim
-C(1-r)^{1+\frac{1}{m}}.
\]
Hence
\[
|W_0(x)|\sim C(1-r)^{1+\frac{1}{m}},
\]
and therefore
\[
\frac{|W_0|}{h}
\sim
C(1-r).
\]
Thus \(\frac{W_0}{h}\) is bounded near the boundary. We conclude that
\[
V\in L^p(Q_1).
\]

Moreover, \(V\) is not divergence-free. Indeed, near the origin the leading
radial part of \(V\) is
\[
a(t)^{m-1}\frac{\nabla g}{h}
\sim
-C(t)r^{-1+q(1-m)}e_r.
\]
Therefore,
\[
\nabla\cdot V
\sim
-C(t)q(1-m)r^{-2+q(1-m)}
\]
near \(r=0\), which is not identically zero.

We now compute the equation satisfied by \(u\). Since
\[
u_t=a'(t)h,
\qquad
u^m=a(t)^m g,
\]
and
\[
\nabla\cdot(a(t)W_0)
=
a(t)h-a(t)M\delta_0,
\]
we obtain, in the sense of distributions,
\[
\begin{aligned}
u_t-\Delta u^m+\nabla\cdot(uV)
&=
a'h-a^m\Delta g
+
\nabla\cdot\left(\eta a^m\nabla g-aW_0\right)\\
&=
a'h-a^m\Delta g
+
a^m\eta\Delta g
+
a^m\nabla\eta\cdot\nabla g
-a h
+aM\delta_0\\
&=
(a'-a)h
-a^m(1-\eta)\Delta g
+
a^m\nabla\eta\cdot\nabla g
+
aM\delta_0.
\end{aligned}
\]
Thus
\[
u_t-\Delta u^m+\nabla\cdot(uV)
=
f+a(t)M\delta_0\otimes dt,
\]
where \(\delta_0\otimes dt\) denotes the product of the spatial Dirac mass at the origin and the Lebesgue measure in time, and
\[
f:=
(a'-a)h
-a^m(1-\eta)\Delta g
+
a^m\nabla\eta\cdot\nabla g.
\]

We claim that
\[
f\ge0,
\qquad
f\in L^1(Q_1).
\]
First,
\[
a(t)=e^t t^\gamma
\]
implies
\[
a'(t)-a(t)
=
e^t\gamma t^{\gamma-1}\ge0.
\]
Hence
\[
(a'-a)h\ge0.
\]
Next, the term \(1-\eta\) is supported where \(r\ge R\), and on this region
\[
\Delta g\le0.
\]
Therefore
\[
-a^m(1-\eta)\Delta g\ge0.
\]
Finally, since both \(\eta\) and \(g\) are radial,
\[
\nabla\eta\cdot\nabla g=\eta'(r)g'(r).
\]
But
\[
\eta'(r)\le0,
\qquad
g'(r)=-sr^{-s-1}-Lr^{L-1}<0.
\]
Thus
\[
\nabla\eta\cdot\nabla g\ge0.
\]
Consequently,
\[
f\ge0.
\]

The integrability of \(f\) is also straightforward. Since
\[
a'(t)-a(t)=e^t\gamma t^{\gamma-1}\in L^1(0,1)
\]
and \(h\in L^1(D_1)\), we have
\[
(a'-a)h\in L^1(Q_1).
\]
The remaining terms are supported away from the origin, where \(g\),
\(\nabla g\), and \(\Delta g\) are bounded. Hence
\[
f\in L^1(Q_1).
\]
Therefore
\[
\mu:=f\,dxdt+a(t)M\delta_0\otimes dt
\]
is a nonnegative finite Radon measure on \(Q_1\).
Here \(\delta_0\otimes dt\) denotes the product of the spatial Dirac mass at the origin and the Lebesgue measure in time.

Finally, we show that the gradient estimate fails. Near \(r=0\),
\[
u^{\frac{m}{2}}
=
a(t)^{\frac{m}{2}}h(r)^{\frac{m}{2}}
\sim
a(t)^{\frac{m}{2}}r^{-\frac{qm}{2}}.
\]
Thus
\[
\left|\nabla u^{\frac{m}{2}}\right|
\sim
a(t)^{\frac{m}{2}}r^{-1-\frac{qm}{2}}.
\]
For \(\alpha>0\), the spatial integral near the origin behaves like
\[
\int_0 r^{1-\alpha\left(1+\frac{qm}{2}\right)}\,dr.
\]
This integral diverges whenever
\[
1-\alpha\left(1+\frac{qm}{2}\right)\le -1
\quad \Longleftrightarrow \quad
\alpha\ge \frac{2}{1+\frac{qm}{2}}.
\]
Since \(q>0\) and \(m>0\), the denominator $1+\frac{qm}{2}$ is larger than \(1\), and hence we may choose
\[
\alpha_0:=
\frac{2}{1+\frac{qm}{2}}\in(0,2).
\]
Therefore,
\[
\int_{Q_1}
\left|\nabla u^{\frac{m}{2}}\right|^{\alpha_0}\,dxdt
=
\infty.
\]
\end{example}

By construction, the drift field \(V\) belongs to the same isotropic integrability scale as the divergence-free admissible class in Theorem~\ref{T:divfree}. Nevertheless, the gradient estimate fails because the drift is not divergence-free. This shows that the cancellation coming from \(\nabla\cdot V=0\) is
essential. 

In particular, below the general-drift threshold
\(
m=1-\frac{1}{d}=\frac{1}{2}
\)
in two dimensions, the estimates obtained in the divergence-free case
cannot be expected for arbitrary drift fields under the same integrability
assumptions on \(V\). Therefore, the condition
\[
m>1-\frac{1}{d}
\]
in Theorem~\ref{T:PME} is sharp with respect to general, non-divergence-free
drifts.

Consequently, the threshold \(m>1-\frac{1}{d}\) in the general drift case is
not only a limitation of our compactness argument, but is sharp in the
following sense: below this threshold, even a drift field satisfying the
corresponding divergence-free integrability scale may destroy the gradient
estimate unless the cancellation condition \(\nabla\cdot V=0\) is imposed.

Together with the sharpness in the drift-free case, this gives the following picture: the lower bound
\(
m>\left(1-\frac{2}{d}\right)_+
\)
is intrinsic to the measure data problem itself, whereas the stronger threshold \(m>1-\frac{1}{d}\) is unavoidable for general drift fields. The divergence-free condition is precisely the structure that allows one to pass from the general-drift threshold to the optimal measure-data threshold.

\subsection{A further obstruction in the range \(m>2\)}
\label{S:obstruction-m-greater-than-two}

The preceding example concerns the sharpness of the lower threshold
\(m>1-\frac{1}{d}\) for general drift fields. We now record a different
phenomenon in the highly degenerate range \(m>2\).

In the proof of Theorem~\ref{T:PME}, the upper restriction \(m\le2\) is
connected with the direct compactness argument for \(u\). Indeed, when
\(m>2\), the factor
\(
u^{1-\frac{m}{2}}
\)
becomes singular near the zero set of \(u\), and estimates for
\(\nabla u^{\frac{m}{2}}\) no longer directly yield compactness of \(u\).

Before giving the example, let us explain the choice of the integrability
condition on the drift. In dimension \(d=2\), the divergence-free admissible
condition \eqref{sub_S_divfree}, in the isotropic case \((q_1,q_2)=(p,p)\),
is exactly \eqref{isotropic_p}. Thus, if a drift field
\(V\in L^p_{x,t}\) with \(p>\frac{m+1}{m}\) were divergence-free, then it
would satisfy the corresponding isotropic admissible condition in
Theorem~\ref{T:divfree}.

The example below chooses \(p\) in this same admissible range, more precisely
\(
\frac{m+1}{m}<p<m,
\)
and constructs a non-divergence-free drift field
\(
V\in L^p(Q_1).
\)
The additional upper bound \(p<m\) is not part of the divergence-free
admissible condition; it is needed in the construction to make the
integrability of \(V\) compatible with the failure of the gradient estimate.
Thus the obstruction is not caused by insufficient integrability of \(V\),
but by the absence of the cancellation structure \(\nabla\cdot V=0\).

\begin{example}
\label{ex:m-greater-than-two-Lp-drift}
Let
\[
D_1:=\{x\in\mathbb R^2:|x|<1\},
\qquad
Q_1:=D_1\times(0,1),
\]
and assume that
\[
m>2.
\]
Let \(p\) satisfy
\[
\frac{m+1}{m}<p<m.
\]
As explained above, the lower bound is the isotropic divergence-free admissibility condition, while the upper bound is used in the construction.
Choose
\[
\frac{m(p-1)}{2p(m-1)}
<
\vartheta
<
\frac{1}{2}.
\]
This is possible precisely because \(p<m\).

Let \(\chi\in C_c^\infty(D_1)\) be a nonnegative cutoff function such that
\(
0\le \chi\le1
\), and 
\(
\chi\equiv1
\) in \(
D_{\frac{1}{2}}.
\)
Let \(\gamma>0\), and define
\[
u(x,t)
:=
t^\gamma
\left(
|x_1|^{2\vartheta}\chi(x)^2
\right)^{\frac{1}{m}}.
\]
Then
\(
u(\cdot,0)=0
\) in \(
L^1(D_1),
\)
and \(u\) has zero lateral boundary data, since \(\chi\) is compactly supported
in \(D_1\). Moreover,
\[
u^m(x,t)
=
t^{\gamma m}|x_1|^{2\vartheta}\chi(x)^2.
\]

Define the drift field by
\[
V(x,t):=
\begin{cases}
\dfrac{\nabla u^m(x,t)}{u(x,t)}, & u(x,t)>0,\\[2mm]
0, & u(x,t)=0.
\end{cases}
\]
Equivalently, on \(\{u>0\}\),
\[
V(x,t)
=
t^{\gamma(m-1)}
\frac{
\nabla\left(|x_1|^{2\vartheta}\chi(x)^2\right)
}{
\left(|x_1|^{2\vartheta}\chi(x)^2\right)^{\frac{1}{m}}
}.
\]
By construction,
\[
uV=\nabla u^m
\quad\text{a.e. in }Q_1.
\]
Therefore, %in the sense of distributions,
%\[
%u_t-\Delta u^m+\nabla\cdot(uV)
%=
%u_t.
%\]
%Hence 
\(u\) solves
\[
u_t-\Delta u^m+\nabla\cdot(uV)=\mu
\quad\text{in } Q_1
\]
in the sense of distributions,
where
\[
\mu
:=
\gamma t^{\gamma-1}
\left(
|x_1|^{2\vartheta}\chi(x)^2
\right)^{\frac{1}{m}}.
\]
Since
\[
\gamma t^{\gamma-1}\ge0,
\qquad
\int_0^1 t^{\gamma-1}\,dt<\infty,
\qquad
\left(
|x_1|^{2\vartheta}\chi(x)^2
\right)^{\frac{1}{m}}
\in L^1(D_1),
\]
the measure \(\mu\) is a nonnegative finite Radon measure on \(Q_1\).

We now verify that
\[
V\in L^p(Q_1).
\]
Near the set \(\{x_1=0\}\) and inside \(D_{\frac{1}{2}}\), where
\(\chi\equiv1\), we have
\[
u(x,t)=t^\gamma |x_1|^{\frac{2\vartheta}{m}},
\qquad
u^m(x,t)=t^{\gamma m}|x_1|^{2\vartheta}.
\]
Thus
\[
|V(x,t)|
\sim
t^{\gamma(m-1)}
|x_1|^{-1+\frac{2\vartheta(m-1)}{m}}
\]
near \(\{x_1=0\}\).
Indeed,
\[
\nabla\left(|x_1|^{2\vartheta}\chi^2\right)
=
2\vartheta\,\operatorname{sgn}(x_1)
|x_1|^{2\vartheta-1}\chi(x)^2 e_1
+
2|x_1|^{2\vartheta}\chi(x)\nabla\chi(x),
\]
which implies
\[
\left|
\frac{
\nabla\left(|x_1|^{2\vartheta}\chi(x)^2\right)
}{
\left(|x_1|^{2\vartheta}\chi(x)^2\right)^{\frac{1}{m}}
}
\right|
\le
C|x_1|^{-1+\frac{2\vartheta(m-1)}{m}}
+
C,
\]
where we used \(m>2\), so that the power 
\(
\chi^{1-\frac{2}{m}}
\)
appearing in the cutoff term is bounded.

Since
\(
t^{\gamma(m-1)}\in L^p(0,1),
\)
it remains to check the spatial singularity. We need
\[
\int_{-\varepsilon}^{\varepsilon}
|x_1|^{p\left(-1+\frac{2\vartheta(m-1)}{m}\right)}
\,dx_1<\infty.
\]
This is equivalent to
\[
p\left(1-\frac{2\vartheta(m-1)}{m}\right)<1 \quad \iff \quad
\vartheta>
\frac{m(p-1)}{2p(m-1)}.
\]
This is exactly our choice of \(\vartheta\). Hence
\(
V\in L^p(Q_1).
\)
Moreover, by construction, \(V\) is not divergence-free, since its leading component near \(\{x_1=0\}\) depends nontrivially on \(x_1\). %Indeed, in \(D_{\frac{1}{2}}\) near \(\{x_1=0\}\), the leading component of \(V\) is in the \(x_1\)-direction and has the form
%\[
%V_1(x,t)
%\sim
%C(t)\operatorname{sgn}(x_1)
%|x_1|^{-1+\frac{2\vartheta(m-1)}{m}}.
%\]
%Its derivative in the \(x_1\)-variable is not identically zero. Thus
%\[
%\nabla\cdot V\not\equiv0.
%\]

Finally, we show that the gradient estimate fails. Since
\[
u^{\frac{m}{2}}
=
t^{\frac{\gamma m}{2}}
\left(
|x_1|^{2\vartheta}\chi(x)^2
\right)^{\frac{1}{2}}
=
t^{\frac{\gamma m}{2}}
|x_1|^\vartheta\chi(x),
\]
and since \(\chi\equiv1\) in \(D_{\frac{1}{2}}\), near \(\{x_1=0\}\) we have
\(
u^{\frac{m}{2}}
\sim
t^{\frac{\gamma m}{2}}
|x_1|^\vartheta.
\)
Therefore
\[
\left|\nabla u^{\frac{m}{2}}\right|
\sim
t^{\frac{\gamma m}{2}}
|x_1|^{\vartheta-1}.
\]
The spatial integral in the \(x_1\)-variable behaves like
\[
\int_{-\varepsilon}^{\varepsilon}
|x_1|^{\alpha(\vartheta-1)}\,dx_1.
\]
This integral diverges whenever
\(
\alpha(1-\vartheta)\ge1.
\)
Choose
\(
\alpha_0:=\frac{1}{1-\vartheta}.
\)
Since \(\vartheta<\frac{1}{2}\), we have
\(
1<\alpha_0<2.
\)
Consequently,
\[
\int_{Q_1}
\left|\nabla u^{\frac{m}{2}}\right|^{\alpha_0}
\,dxdt
=
\infty.
\]
\end{example}

The restriction \(p<m\) in the preceding construction is not accidental.
The condition \(V\in L^p(Q_1)\) requires
\(
\vartheta>
\frac{m(p-1)}{2p(m-1)},
\)
whereas the failure of the gradient estimate with some exponent
\(\alpha_0<2\) requires
\(
\vartheta<\frac{1}{2}.
\)
These two inequalities are compatible precisely when \(p<m\). Thus the
construction gives examples for every
\[
\frac{m+1}{m}<p<m,
\]
and in particular for \(p=2\) whenever \(m>2\).

The example shows that, without the divergence-free structure, the drift can
interact with the degenerate diffusion in a much stronger way than in the
divergence-free case. Indeed, the drift is chosen so that
\[
uV=\nabla u^m,
\]
and hence the drift contribution cancels the diffusion term. Therefore the
gradient estimate
\[
\nabla u^{\frac{m}{2}}\in L^\alpha(Q_1)
\quad
\text{for every }0<\alpha<2
\]
cannot be expected for arbitrary non-divergence-free drifts, even in the
range \(m>2\), under the above integrability condition on \(V\).

This is a different obstruction from the one in
Example~\ref{ex:non-divfree-counterexample}. The first example proves the
sharpness of the lower threshold \(m>1-\frac{1}{d}\) for general drifts. The
present example shows that, in the highly degenerate range \(m>2\), the
divergence-free cancellation is also essential for recovering the gradient
estimate available in Theorem~\ref{T:divfree}.

%%%%%%%%%%%%%%%%%%%%%%%%%%%%%%%%%%%
\section{Application}\label{S:Application}

In this section, we study, as an application, a porous medium equation with measure data coupled to the incompressible Navier--Stokes equations. More precisely, we consider
\begin{equation}\label{KF-10}
\rho_t-\Delta \rho^m +v\cdot \nabla \rho=\mu,
\end{equation}
\begin{equation}\label{KF-20}
v_t-\Delta v +(v\cdot \nabla) v+\nabla \pi=-\rho\nabla\phi,
\qquad {\rm div}\,v=0
\end{equation}
in $\Omega_T=\Omega\times (0, T)$, where $\Omega\subset\mathbb R^d$, $d=2,3$, is a bounded domain with smooth boundary, subject to the homogeneous Dirichlet boundary conditions
\[
\rho=0,\qquad v=0
\qquad \text{on } \partial\Omega\times(0,T).
\]
Here $\phi=\phi(x)$ is a given sufficiently regular potential, for instance a gravitational potential, and the initial data are assumed to satisfy
\[
\rho_0\in L^1(\Omega),
\qquad
v_0\in L^2(\Omega).
\]

The system \eqref{KF-10}--\eqref{KF-20} can be viewed as a simplified Keller--Segel--fluid type system with porous medium diffusion and measure data, namely
\[
\rho_t-\Delta \rho^m +v\cdot \nabla \rho+\chi\nabla\cdot(\rho\nabla c)=\mu,
\]
\[
c_t-\Delta c +v\cdot \nabla c=-c\rho,
\]
\[
v_t-\Delta v +(v\cdot \nabla) v+\nabla \pi=-\rho\nabla\phi,
\qquad {\rm div}\,v=0.
\]
If the chemical concentration is absent, that is, if $c=0$, then the above system reduces to \eqref{KF-10}--\eqref{KF-20}. In the case $\mu=0$, there are already several results on existence and regularity for Keller--Segel--fluid systems with porous medium diffusion; see, for example, \cite{CHKK17, HKK02, HZ21, Win15} and the references therein.

Our goal is to construct weak solutions to \eqref{KF-10}--\eqref{KF-20}. We first introduce the notion of weak solutions.

\begin{definition}\label{Weak-ks} Let $m>0$ and $\Omega\subset\R^d$, $d=2,3$ be a bounded domain with smooth boundary. We say that a pair of $(\rho, v)$ is a  weak solution of \eqref{KF-10}--\eqref{KF-20} in $\Omega_{T} := \Omega \times (0, T)$ with zero boundary data if the following are satisfied:
\begin{itemize}
	\item [(i)] It holds that
	\[
	\rho , \ \nabla \rho^m, \ \rho v, \, v\otimes v,\, \nabla v\in L_{x,t}^{1}.
	\]
	
	\item [(ii)] For any $\varphi \in C_{c}^{\infty} \left(\overline{\Omega} \times [0, T);\R\right)$ and $\psi \in C_{c}^{\infty} \left(\overline{\Omega} \times [0, T);\R^d\right)$ with $\rm{div}\,\psi=0$ both of which vanish on $\partial\Omega$, it holds that
	\[
	\int_{\Omega_T}\left\{-\rho\,\varphi_t+\nabla \rho^m\cdot \nabla\varphi-\rho v\cdot \nabla\varphi\right\}\,dxdt
	=
	\int_\Omega \rho_0\,\varphi(\cdot,0)\,dx
	+
	\int_{\Omega_T}\varphi\,d\mu,
	\]
	\[
	\int_{\Omega_T}\left\{-v\cdot\psi_t+\nabla v:\nabla\psi-(v\otimes v):\nabla\psi\right\}\,dxdt
	=
	\int_\Omega v_0\cdot\psi(\cdot,0)\,dx
	-
	\int_{\Omega_T}\rho\,\nabla\phi\cdot\psi\,dxdt.
	\]
\end{itemize}	
\end{definition}

To construct a weak solution, it is required to  develop an approximated system of \eqref{KF-10}--\eqref{KF-20}. Since the argument of regularization is somewhat similar as mentioned earlier in Section~\ref{S:Exist}, we provide a priori estimate in the proof for construction of weak solutions of \eqref{KF-10}--\eqref{KF-20}.

\begin{theorem}
Let  $\Omega\subset\R^d$, $d=2,3$ be a bounded domain with a smooth boundary. Assume that $\rho_0\in L^1(\Omega)$, $v_0\in L^2(\Omega)$, and let $\mu$ be a nonnegative finite measure on $\Omega_T$.
 %Suppose that $V \in \mathfrak{S}_{m, \sigma}^{(q_1, q_2)}$  for  $\  q_2 >2$.
If $m>0$ when $d=2$ and if $m>\frac{2}{3}$ when $d=3$, then
%Let $1-\frac{1}{d} < m \le 2$ and $\Omega\subset\R^d$, $d=2,3$ be a bounded domain with smooth boundaries. %Suppose that $V \in \mathfrak{S}_{m, \sigma}^{(q_1, q_2)}$  for  $\  q_2 >2$.
there exists a pair of weak solutions $(\rho, v)$ in Definition \ref{Weak-ks}. Moreover, for any $0<\alpha<2$, a pair of weak solutions $(\rho, v)$ satisfies
\[
\sup_{t\in (0, T)} \int_{\Omega\times\{t\}} (\rho(\cdot, t) +\abs{v(\cdot, t)}^2)\,dx +\int_{\Omega_T} (\left|\nabla \rho^{\frac{m}{2}}\right|^{\alpha} +\abs{\nabla v}^2)\,dxdt\le C,
\]
where
\[
C=C\!\left(m,d,\alpha,\Omega,T,\|\nabla\phi\|_{L^\infty},
\|\rho_0\|_{L^1(\Omega)},\|v_0\|_{L^2(\Omega)},\mu(\Omega_T)\right).
\]
\end{theorem}

\begin{proof}
We only explain the core a priori estimates for smooth solutions.  
The standard regularization and compactness argument producing a weak solution pair is exactly the same as in Section~\ref{S:Exist}, and is therefore omitted.

We set
$\nu:=\|\rho_0\|_{L^1(\Omega)}+\mu(\Omega_T)$.
Since $v$ is divergence-free, Theorem~\ref{T:divfree}, together with Remark~\ref{L0 rmk}, applies to the equation \eqref{KF-10}. In particular, for every $0<\alpha<2$,
\begin{equation}\label{app-rho-general}
\sup_{t\in(0,T)}\int_\Omega \rho(\cdot,t)\,dx
+
\int_{\Omega_T}\left|\nabla \rho^{\frac m2}\right|^\alpha\,dxdt
\le
C(\nu),
\end{equation}
where the constant is independent of the drift $v$.

We now estimate the fluid component.

\smallskip
\noindent
$\bullet$  \underline{\emph{(2d case)}} Assume that $d=2$ and $m>0$.
Choose
\(
\bar\alpha\in
\left(
\max\left\{1,\frac{2}{m+1}\right\},2
\right).
\)
Then \eqref{app-rho-general} yields
\begin{equation}\label{app-2d-rho}
\sup_{t\in(0,T)}\int_\Omega \rho(\cdot,t)\,dx
+
\int_{\Omega_T}\left|\nabla \rho^{\frac m2}\right|^{\bar\alpha}\,dxdt
\le
C(\nu).
\end{equation}

Testing \eqref{KF-20} against $v$, and using $\operatorname{div}v=0$, we obtain
\[
\frac12\frac{d}{dt}\|v\|_{L^2(\Omega)}^2+\|\nabla v\|_{L^2(\Omega)}^2
=
-\int_\Omega \rho\,\nabla\phi\cdot v\,dx
\le
\|\nabla\phi\|_{L^\infty}\|\rho\|_{L^q}\|v\|_{L^p},
\]
where
$\frac1p+\frac1q=1$.
By the two-dimensional Gagliardo--Nirenberg inequality,
\[
\|v\|_{L^p}
\le
C\|v\|_{L^2}^{\frac{2}{p}}\|\nabla v\|_{L^2}^{1-\frac{2}{p}}.
\]
Moreover, by interpolation between $L^1(\Omega)$ and the Sobolev exponent for $\rho^{m/2}$,
\[
\|\rho\|_{L^q}
\le
C
\|\rho\|_{L^1}^{\theta}
\|\nabla \rho^{\frac m2}\|_{L^{\bar\alpha}}^{\frac{2}{m}(1-\theta)},
\qquad
\theta=\frac{(m+q)\bar\alpha-2q}{q(\bar\alpha(m+1)-2)}.
\]
Substituting these bounds into the energy inequality for $v$ and applying Young's inequality, we obtain
\[
\frac{d}{dt}\|v\|_{L^2}^2+\|\nabla v\|_{L^2}^2
\le
\varepsilon \|\nabla \rho^{\frac m2}\|_{L^{\bar\alpha}}^{\bar\alpha}
+
C_\varepsilon\left(1+\|v\|_{L^2}^2\right),
\]
provided that $p>2$ is chosen sufficiently large. Since \eqref{app-2d-rho} controls
$\|\nabla \rho^{\frac m2}\|_{L^{\bar\alpha}(\Omega_T)}$,
Gronwall's inequality yields
\begin{equation}\label{app-2d-v}
\sup_{t\in(0,T)}\|v(\cdot,t)\|_{L^2(\Omega)}^2
+
\int_{\Omega_T}|\nabla v|^2\,dxdt
\le
C.
\end{equation}
Combining \eqref{app-2d-rho} and \eqref{app-2d-v}, and then using \eqref{app-rho-general} again for arbitrary $0<\alpha<2$, we conclude that
\[
\sup_{t\in(0,T)}\int_\Omega \left(\rho(\cdot,t)+|v(\cdot,t)|^2\right)\,dx
+
\int_{\Omega_T}\left(\left|\nabla \rho^{\frac m2}\right|^\alpha+|\nabla v|^2\right)\,dxdt
\le
C
\]
for every $0<\alpha<2$.

\smallskip
\noindent
$\bullet$ \underline{\emph{(3d case)}} Assume that $d=3$ and $m>\frac23$.
Choose
\(
\bar\alpha\in
\left(
\max\left\{1,\frac{8}{3m+2}\right\},2
\right).
\)
Then \eqref{app-rho-general} gives
\begin{equation}\label{app-3d-rho}
\sup_{t\in(0,T)}\int_\Omega \rho(\cdot,t)\,dx
+
\int_{\Omega_T}\left|\nabla \rho^{\frac m2}\right|^{\bar\alpha}\,dxdt
\le
C(\nu).
\end{equation}

Testing \eqref{KF-20} against $v$, and using $\operatorname{div}v=0$, we obtain
\[
\frac12\frac{d}{dt}\|v\|_{L^2(\Omega)}^2+\|\nabla v\|_{L^2(\Omega)}^2
=
-\int_\Omega \rho\,\nabla\phi\cdot v\,dx
\lesssim
\|\rho\|_{L^{6/5}}\|v\|_{L^6}
\le
C\|\rho\|_{L^{6/5}}^2+\frac12\|\nabla v\|_{L^2}^2.
\]
Hence, we get
\begin{equation}\label{app-3d-v1}
\frac12\frac{d}{dt}\|v\|_{L^2}^2+\|\nabla v\|_{L^2}^2
\lesssim
\|\rho\|_{L^{6/5}}^2.
\end{equation}

We next estimate $\|\rho\|_{L^{6/5}}$. By interpolation, together with the uniform $L^1$-bound on $\rho$,
\[
\|\rho\|_{L^{6/5}}^2
\le
\|\rho\|_{L^1}^{2\theta}
\|\rho\|_{L^{\frac{3m\bar\alpha}{2(3-\bar\alpha)}}}^{2(1-\theta)}
\lesssim
\left\|\rho^{\frac m2}\right\|_{L^{\frac{3\bar\alpha}{3-\bar\alpha}}}^{\frac{4(1-\theta)}{m}}
\lesssim
\left\|\nabla \rho^{\frac m2}\right\|_{L^{\bar\alpha}}^{\frac{4(1-\theta)}{m}},
\]
where
$\displaystyle
\theta=\frac{(5m+4)\bar\alpha-12}{2(\bar\alpha(3m+2)-6)}$.
A direct computation gives
\[
\frac{4(1-\theta)}{m}
=
\frac{2\bar\alpha}{\bar\alpha(3m+2)-6}.
\]
Therefore, we have
\begin{equation}\label{app-3d-v2}
\frac{d}{dt}\|v\|_{L^2}^2+\|\nabla v\|_{L^2}^2
\lesssim
\left\|\nabla \rho^{\frac m2}\right\|_{L^{\bar\alpha}}^{\frac{2\bar\alpha}{\bar\alpha(3m+2)-6}}.
\end{equation}
Since
\(
\bar\alpha>\frac{8}{3m+2},
\)
it follows that
\[
\frac{2\bar\alpha}{\bar\alpha(3m+2)-6}\le \bar\alpha.
\]
Combining \eqref{app-3d-v2} with \eqref{app-3d-rho}, we conclude that
\begin{equation}\label{app-3d-v}
\sup_{t\in(0,T)}\|v(\cdot,t)\|_{L^2(\Omega)}^2
+
\int_{\Omega_T}|\nabla v|^2\,dxdt
\le
C.
\end{equation}
Together with \eqref{app-rho-general}, this yields
\[
\sup_{t\in(0,T)}\int_\Omega \left(\rho(\cdot,t)+|v(\cdot,t)|^2\right)\,dx
+
\int_{\Omega_T}\left(\left|\nabla \rho^{\frac m2}\right|^\alpha+|\nabla v|^2\right)\,dxdt
\le
C
\]
for every $0<\alpha<2$.

Finally, the above estimates are stable under the standard regularization procedure, and the compactness argument from Section~\ref{S:Exist} yields a weak solution pair $(\rho,v)$ in the sense of Definition~\ref{Weak-ks}. This completes the proof.
\end{proof}

%%%%%%%%%%%%%%%%%%%%%%%%%%%%%%%%%%%%

%\section*{Conflicts of Interest}
%There is no conflict of interest.

%%%%%%%%%%%%%%%%%%%%%%%%%%%%%%%%%%%%

%\bibliographystyle{amsplain}

%\bibliography{PME.bib}

\begin{bibdiv}
\begin{biblist}

\bib{BedHe2017}{article}{
      author={Bedrossian, J.},
      author={He, S.},
       title={Suppression of blow-up in {P}atlak-{K}eller-{S}egel via shear
  flows},
        date={2017},
        ISSN={0036-1410,1095-7154},
     journal={SIAM J. Math. Anal.},
      volume={49},
      number={6},
       pages={4722\ndash 4766},
         url={https://doi.org/10.1137/16M1093380},
      review={\MR{3730537}},
}

\bib{BDG13}{article}{
      author={B\"{o}gelein, V.},
      author={Duzaar, F.},
      author={Gianazza, U.},
       title={Porous medium type equations with measure data and potential
  estimates},
        date={2013},
        ISSN={0036-1410},
     journal={SIAM J. Math. Anal.},
      volume={45},
      number={6},
       pages={3283\ndash 3330},
         url={https://doi.org/10.1137/130925323},
      review={\MR{3124895}},
}

\bib{BDG14}{article}{
      author={B\"{o}gelein, V.},
      author={Duzaar, F.},
      author={Gianazza, U.},
       title={Continuity estimates for porous medium type equations with
  measure data},
        date={2014},
        ISSN={0022-1236},
     journal={J. Funct. Anal.},
      volume={267},
      number={9},
       pages={3351\ndash 3396},
         url={https://doi.org/10.1016/j.jfa.2014.08.014},
      review={\MR{3261113}},
}

\bib{BDG15}{article}{
      author={B\"{o}gelein, V.},
      author={Duzaar, F.},
      author={Gianazza, U.},
       title={Very weak solutions of singular porous medium equations with
  measure data},
        date={2015},
        ISSN={1534-0392},
     journal={Commun. Pure Appl. Anal.},
      volume={14},
      number={1},
       pages={23\ndash 49},
         url={https://doi.org/10.3934/cpaa.2015.14.23},
      review={\MR{3299023}},
}

\bib{BDG16}{article}{
      author={B\"{o}gelein, V.},
      author={Duzaar, F.},
      author={Gianazza, U.},
       title={Sharp boundedness and continuity results for the singular porous
  medium equation},
        date={2016},
        ISSN={0021-2172},
     journal={Israel J. Math.},
      volume={214},
      number={1},
       pages={259\ndash 314},
         url={https://doi.org/10.1007/s11856-016-1360-3},
      review={\MR{3540615}},
}

\bib{CCY2019}{incollection}{
      author={Carrillo, J.~A.},
      author={Craig, K.},
      author={Yao, Y.},
       title={Aggregation-diffusion equations: dynamics, asymptotics, and
  singular limits},
        date={2019},
   booktitle={Active particles. {V}ol. 2. {A}dvances in theory, models, and
  applications},
      series={Model. Simul. Sci. Eng. Technol.},
   publisher={Birkh\"auser/Springer, Cham},
       pages={65\ndash 108},
      review={\MR{3932458}},
}

\bib{CHKK17}{article}{
      author={Chung, Y.~S.},
      author={Hwang, S.},
      author={Kang, K.},
      author={Kim, J.},
       title={H\"{o}lder continuity of {K}eller-{S}egel equations of porous
  medium type coupled to fluid equations},
        date={2017},
        ISSN={0022-0396},
     journal={J. Differential Equations},
      volume={263},
      number={4},
       pages={2157\ndash 2212},
         url={https://doi.org/10.1016/j.jde.2017.03.042},
      review={\MR{3650336}},
}

\bib{DK07}{book}{
      author={Daskalopoulos, P.},
      author={Kenig, C.~E.},
       title={Degenerate diffusions},
      series={EMS Tracts in Mathematics},
   publisher={European Mathematical Society (EMS), Z\"{u}rich},
        date={2007},
      volume={1},
        ISBN={978-3-03719-033-3},
         url={https://doi.org/10.4171/033},
        note={Initial value problems and local regularity theory},
      review={\MR{2338118}},
}

\bib{DB93}{book}{
      author={DiBenedetto, E.},
       title={Degenerate parabolic equations},
      series={Universitext},
   publisher={Springer-Verlag, New York},
        date={1993},
        ISBN={0-387-94020-0},
         url={https://doi.org/10.1007/978-1-4612-0895-2},
      review={\MR{1230384}},
}

\bib{DGV12}{book}{
      author={DiBenedetto, E.},
      author={Gianazza, U.},
      author={Vespri, V.},
       title={Harnack's inequality for degenerate and singular parabolic
  equations},
      series={Springer Monographs in Mathematics},
   publisher={Springer, New York},
        date={2012},
        ISBN={978-1-4614-1583-1},
         url={https://doi.org/10.1007/978-1-4614-1584-8},
      review={\MR{2865434}},
}

\bib{Freitag}{article}{
      author={Freitag, M.},
       title={Global existence and boundedness in a chemorepulsion system with
  superlinear diffusion},
        date={2018},
        ISSN={1078-0947},
     journal={Discrete Contin. Dyn. Syst.},
      volume={38},
      number={11},
       pages={5943\ndash 5961},
         url={https://doi.org/10.3934/dcds.2018258},
      review={\MR{3917794}},
}

\bib{HKK01}{article}{
      author={Hwang, S.},
      author={Kang, K.},
      author={Kim, H.~K.},
       title={Existence of weak solutions for porous medium equation with a
  divergence type of drift term},
        date={2023},
        ISSN={0944-2669},
     journal={Calc. Var. Partial Differential Equations},
      volume={62},
      number={4},
       pages={Paper No. 126},
         url={https://doi.org/10.1007/s00526-023-02451-4},
      review={\MR{4568176}},
}

\bib{HKK02}{article}{
      author={Hwang, S.},
      author={Kang, K.},
      author={Kim, H.~K.},
       title={Existence of weak solutions for porous medium equation with a
  divergence type of drift term in a bounded domain},
        date={2024},
        ISSN={0022-0396,1090-2732},
     journal={J. Differential Equations},
      volume={389},
       pages={361\ndash 414},
         url={https://doi.org/10.1016/j.jde.2024.01.028},
      review={\MR{4697984}},
}

\bib{HKK_FDE}{article}{
      author={Hwang, S.},
      author={Kang, K.},
      author={Kim, H.~K.},
       title={Existence of weak solutions for fast diffusion equation with a
  divergence type of drift term},
     journal={Preprint, arXiv:2501.09539 [math.AP]},
}

\bib{HZ21}{article}{
      author={Hwang, S.},
      author={Zhang, Y.~P.},
       title={Continuity results for degenerate diffusion equations with
  {$L_t^p L_x^q$} drifts},
        date={2021},
        ISSN={0362-546X},
     journal={Nonlinear Anal.},
      volume={211},
       pages={Paper No. 112413, 37},
         url={https://doi.org/10.1016/j.na.2021.112413},
      review={\MR{4265723}},
}

\bib{KK-SIMA}{article}{
      author={Kang, K.},
      author={Kim, H.~K.},
       title={Existence of weak solutions in {W}asserstein space for a
  chemotaxis model coupled to fluid equations},
        date={2017},
        ISSN={0036-1410},
     journal={SIAM J. Math. Anal.},
      volume={49},
      number={4},
       pages={2965\ndash 3004},
         url={https://doi.org/10.1137/16M1083232},
      review={\MR{3682182}},
}

\bib{KZ18}{article}{
      author={Kim, I.},
      author={Zhang, Y.~P.},
       title={Regularity properties of degenerate diffusion equations with
  drifts},
        date={2018},
        ISSN={0036-1410},
     journal={SIAM J. Math. Anal.},
      volume={50},
      number={4},
       pages={4371\ndash 4406},
         url={https://doi.org/10.1137/17M1159749},
      review={\MR{3842921}},
}

\bib{KZ2021}{article}{
      author={Kim, I.},
      author={Zhang, Y.~P.},
       title={Porous medium equation with a drift: free boundary regularity},
        date={2021},
        ISSN={0003-9527,1432-0673},
     journal={Arch. Ration. Mech. Anal.},
      volume={242},
      number={2},
       pages={1177\ndash 1228},
         url={https://doi.org/10.1007/s00205-021-01702-y},
      review={\MR{4331024}},
}

\bib{KLLP19}{article}{
      author={Kinnunen, J.},
      author={Lehtel\"{a}, P.},
      author={Lindqvist, P.},
      author={Parviainen, M.},
       title={Supercaloric functions for the porous medium equation},
        date={2019},
        ISSN={1424-3199},
     journal={J. Evol. Equ.},
      volume={19},
      number={1},
       pages={249\ndash 270},
         url={https://doi.org/10.1007/s00028-018-0474-y},
      review={\MR{3918522}},
}

\bib{LS13}{article}{
      author={Liskevich, V.},
      author={Skrypnik, I.~I.},
       title={Pointwise estimates for solutions to the porous medium equation
  with measure as a forcing term},
        date={2013},
        ISSN={0021-2172},
     journal={Israel J. Math.},
      volume={194},
      number={1},
       pages={259\ndash 275},
         url={https://doi.org/10.1007/s11856-012-0098-9},
      review={\MR{3047070}},
}

\bib{Luk10}{article}{
      author={Lukkari, T.},
       title={The porous medium equation with measure data},
        date={2010},
        ISSN={1424-3199},
     journal={J. Evol. Equ.},
      volume={10},
      number={3},
       pages={711\ndash 729},
         url={https://doi.org/10.1007/s00028-010-0067-x},
      review={\MR{2674065}},
}

\bib{Luk12}{article}{
      author={Lukkari, T.},
       title={The fast diffusion equation with measure data},
        date={2012},
        ISSN={1021-9722},
     journal={NoDEA Nonlinear Differential Equations Appl.},
      volume={19},
      number={3},
       pages={329\ndash 343},
         url={https://doi.org/10.1007/s00030-011-0131-4},
      review={\MR{2926301}},
}

\bib{SSSZ}{article}{
      author={Seregin, G.},
      author={Silvestre, L.},
      author={\v{S}ver\'{a}k, V.},
      author={Zlato\v{s}, A.},
       title={On divergence-free drifts},
        date={2012},
        ISSN={0022-0396},
     journal={J. Differential Equations},
      volume={252},
      number={1},
       pages={505\ndash 540},
         url={https://doi.org/10.1016/j.jde.2011.08.039},
      review={\MR{2852216}},
}

\bib{SVZ13}{article}{
      author={Silvestre, L.},
      author={Vicol, V.},
      author={Zlato\v{s}, A.},
       title={On the loss of continuity for super-critical drift-diffusion
  equations},
        date={2013},
        ISSN={0003-9527},
     journal={Arch. Ration. Mech. Anal.},
      volume={207},
      number={3},
       pages={845\ndash 877},
         url={https://doi.org/10.1007/s00205-012-0579-3},
      review={\MR{3017289}},
}

\bib{Sim87}{article}{
      author={Simon, J.},
       title={Compact sets in the space {$L^p(0,T;B)$}},
        date={1987},
        ISSN={0003-4622},
     journal={Ann. Mat. Pura Appl. (4)},
      volume={146},
       pages={65\ndash 96},
         url={https://doi.org/10.1007/BF01762360},
      review={\MR{916688}},
}

\bib{Stu15}{article}{
      author={Sturm, S.},
       title={Pointwise estimates for porous medium type equations with low
  order terms and measure data},
        date={2015},
     journal={Electron. J. Differential Equations},
       pages={No. 101, 25 pp},
      review={\MR{3358473}},
}

\bib{Vaz06}{book}{
      author={V\'{a}zquez, J.~L.},
       title={Smoothing and decay estimates for nonlinear diffusion equations},
      series={Oxford Lecture Series in Mathematics and its Applications},
   publisher={Oxford University Press, Oxford},
        date={2006},
      volume={33},
        ISBN={978-0-19-920297-3; 0-19-920297-4},
         url={https://doi.org/10.1093/acprof:oso/9780199202973.001.0001},
        note={Equations of porous medium type},
      review={\MR{2282669}},
}

\bib{Vaz07}{book}{
      author={V\'{a}zquez, J.~L.},
       title={The porous medium equation},
      series={Oxford Mathematical Monographs},
   publisher={The Clarendon Press, Oxford University Press, Oxford},
        date={2007},
        ISBN={978-0-19-856903-9; 0-19-856903-3},
        note={Mathematical theory},
      review={\MR{2286292}},
}

\bib{Win15}{article}{
      author={Winkler, M.},
       title={Boundedness and large time behavior in a three-dimensional
  chemotaxis-{S}tokes system with nonlinear diffusion and general sensitivity},
        date={2015},
        ISSN={0944-2669},
     journal={Calc. Var. Partial Differential Equations},
      volume={54},
      number={4},
       pages={3789\ndash 3828},
         url={https://doi.org/10.1007/s00526-015-0922-2},
      review={\MR{3426095}},
}

\bib{Z11}{article}{
      author={Zlato\v{s}, A.},
       title={Reaction-diffusion front speed enhancement by flows},
        date={2011},
        ISSN={0294-1449},
     journal={Ann. Inst. H. Poincar\'{e} C Anal. Non Lin\'{e}aire},
      volume={28},
      number={5},
       pages={711\ndash 726},
         url={https://doi.org/10.1016/j.anihpc.2011.05.004},
      review={\MR{2838397}},
}

\end{biblist}
\end{bibdiv}


\end{document}